\documentclass{amsart}
\usepackage{amssymb}
\usepackage{color}
\usepackage{float}
\usepackage[all,cmtip]{xy}
\usepackage{graphicx}
\usepackage{hyperref}
\usepackage{cancel}
\usepackage{pstricks}
\usepackage{mathabx}

\newcommand{\Acal}{{\mathcal{A}}}

\newcommand{\Ecal}{{\mathcal{E}}}

\newcommand{\Hcal}{{\mathcal{H}}}

\newcommand{\Mcal}{{\mathcal{M}}}
\newcommand{\Ncal}{{\mathcal{N}}}

\newcommand{\Pcal}{{\mathcal{P}}}

\newcommand{\Scal}{{\mathcal{S}}}

\newcommand{\acal}{\mathfrak{a}}
\newcommand{\pacal}{\mathfrak{pa}}
\newcommand{\tacal}{{\mathfrak{ta}}}

\newcommand{\CC}{\mathbb{C}}

\newcommand{\NN}{\mathbb{N}}

\newcommand{\RR}{\mathbb{R}}

\newcommand{\ZZ}{\mathbb{Z}}

\newcommand{\id}{\textbf{\textit{I}}}
\newcommand{\im}{\operatorname{im}}

\newcommand{\Vol}{\operatorname{Vol}}

\newcommand{\FH}{\operatorname{FH}}
\newcommand{\FC}{\operatorname{FC}}
\renewcommand{\TH}{\operatorname{TH}}
\newcommand{\TC}{\operatorname{TC}}
\newcommand{\AH}{\operatorname{AH}}
\newcommand{\AC}{\operatorname{AC}}
\newcommand{\PH}{\operatorname{PH}}

\newcommand{\WslantH}{\operatorname{WH}}
\newcommand{\WcheckH}{\operatorname{\check wH}}

\newcommand{\Crit}{\operatorname{Crit}}

\renewcommand{\id}{\operatorname{id}}

\newcommand{\haat}{\widehat}
\newcommand{\tiilde}{\widetilde}

\newtheorem{theorem}{Theorem}
\newtheorem*{theorem*}{Theorem}
\newtheorem{proposition}{Proposition}[section]
\newtheorem{lemma}[proposition]{Lemma}

\newtheorem{corollary}[theorem]{Corollary}

\theoremstyle{definition}
\newtheorem{definition}[proposition]{Definition}
\newtheorem{example}[proposition]{Example}

\newtheorem{assumption}[proposition]{Assumption}

\theoremstyle{remark}
\newtheorem{remark}[proposition]{Remark}

\numberwithin{equation}{section}

\begin{document}

\title[]{Positive topological entropy of positive contactomorphisms}

\author{Lucas Dahinden}
\address{Universit\'e de Neuch\^atel (UNINE)}
\curraddr{L. Dahinden, Dept. Mathematiques Rue Emile-Argand 2, 2000 Neuchatel}
\email{l.dahinden@gmail.com}
\thanks{The first author was supported by SNF grant 200021-163419/1.}


\date{\today}

\begin{abstract}
A positive contactomorphism of a contact manifold $M$ is the end point of a contact isotopy on $M$ that is always positively transverse to the contact structure. 
Assume that $M$ contains a Legendrian sphere $\Lambda$, and that $(M,\Lambda)$ is fillable by a Liouville domain $(W,\omega)$ with exact Lagrangian $L$ such that $\omega|_{\pi_2(W,L)}=0$. 
We show that if the exponential growth of the action filtered wrapped Floer homology of $(W,L)$ is positive, then every positive contactomorphism of $M$ has positive topological entropy. This result generalizes the result of Alves and Meiwes from Reeb flows to positive contactomorphisms, and it yields many examples of contact manifolds on which every positive contactomorphism has positive topological entropy, among them the exotic contact spheres found by Alves and Meiwes.
	
	A main step in the proof is to show that wrapped Floer homology is isomorphic to the positive part of Lagrangian Rabinowitz--Floer homology. 
\end{abstract}

\maketitle

\section{Introduction and results}\label{sec:intro}

An important problem in the study of dynamical systems is to understand 
the complexity of the mappings in question.
A good numerical measure for complexity is the topological entropy.
Consider a compact manifold $M$ and a class of diffeomorphisms~$\mathcal D$
of~$M$. A much studied question is whether a generic map from~$\mathcal D$
has positive topological entropy. 
A very different question is whether {\it every}\/ map from~$\mathcal D$
has positive topological entropy. 
This latter question is only interesting under further assumptions on~$M$ and~$\mathcal D$.
Here, we assume that $M$ is a compact manifold endowed with a contact structure~$\xi$, namely a completely
non-integrable distribution of hyperplanes in the tangent bundle~$TM$.
If we also assume that $\xi$ is co-orientable, namely that there exists a 1-form~$\alpha$ on~$M$
with $\xi = \ker \alpha$, then associated with every choice of such an~$\alpha$ there is a natural flow
generated by the vector field~$R_\alpha$ implicitly defined
by the two equations
$$
d\alpha (R_{\alpha}, \cdot) = 0, \qquad \alpha (R_\alpha) = 1.
$$
Such flows are called Reeb flows of~$\alpha$.
They arise as the restriction of many classical Hamiltonian systems to fixed energy levels.
In particular, geodesic flows are Reeb flows.

The first result on positive topological entropy of all Reeb flows on a class of contact manifolds
was obtained by Macarini--Schlenk in~\cite{MacSch11}, who generalized previous results by Dinaburg, Paternain--Petean 
and Gromov on geodesic flows:
Every Reeb flow on the cosphere bundle over a closed manifold~$Q$ has positive topological
entropy, provided that the topology of~$Q$ is ``sufficiently complicated'' 
(for instance, if the fundamental group or the homology of the based loop space of~$Q$ has exponential growth).

Reeb flows form a quite special class of mappings on a contact manifold. 
In~\cite{D18} the above result was generalized  to a much larger class of diffeomorphisms on the same manifolds,
namely to time-dependent Reeb flows.
Let $(M,\xi)$ be a co-oriented contact manifold. 
A smooth path $\varphi_t$ on~$M$ is a time-dependent Reeb flow 
if it is generated by a time-dependent vector field $R_{\alpha_t}$, where each $\alpha_t$
is a contact form for~$\xi$.
There is a more topological perspective on such flows:
They are exactly the positive contact isotopies on~$(M,\xi)$,
namely the isotopies~$\varphi_t$ with $\varphi_0 = \id$ that are everywhere positively transverse to~$\xi$:
$$
\alpha \left( \tfrac{d}{dt} \varphi_t(x) \right) >0
$$ 
for all $t$ and all $x \in M$, 
for one and hence any contact form~$\alpha$ for~$\xi$.

\begin{definition}
{\rm 
A {\it positive contactomorphism}\/ on $(M,\xi)$ is the end point~$\varphi_1$ of a positive contact isotopy on~$\varphi_t$, 
$t \in [0,1]$.
}
\end{definition}

The first results on positive topological entropy of all Reeb flows on contact manifolds different from 
cosphere bundles were given by Alves~\cite{Alv16a,Alv16b,A17} in dimension three.
More recently, Alves and Meiwes~\cite{AM17} constructed many examples of higher dimensional
contact manifolds for which every Reeb flow has positive topological entropy.
In particular, they found on every sphere of dimension at least~seven a contact structure with
this property. 
They asked the natural question whether their results extend from Reeb flows to positive contactomorphisms.
The present paper answers this question in the affirmative.

We will work in the following geometric setting that is further explained in Section~\ref{sec:wfh}.
\begin{assumption}\label{f=0} 
	The pair $(W,L)$ consists of a Liouville domain $(W,\omega,\lambda)$ with compact contact boundary $(M,\xi=\ker\lambda|_M)$ and an asymptotically conical exact Lagrangian~$L$ with connected Legendrian boundary $\Lambda=\partial L$ such that $\lambda|_L=0$, such that $[\omega]|_{\pi_2(W,L)}=0$, and such that $(\lambda,L)$ is regular. Here, regular means that $\bigcup_{t\neq0} \varphi^t_{\lambda|_M}(\Lambda)$ and $\Lambda$ intersect transversely, where $\varphi^t_{\lambda|_M}$ is the Reeb flow of $\lambda|_M$.
\end{assumption}
 
Under this assumption we can define a $\ZZ_2$-vector space $\WslantH(W,L)$, which we call wrapped Floer homology, see Section~\ref{sec:wh} for the definition. This is a filtered homology, thus for every $a$ there is a vector space $\WslantH^a(W,L)$ and a morphism $\iota_a:\WslantH^a(W,L)\to \WslantH(W,L)$. The vector spaces $\WslantH^a(W,L)$ are finite dimensional. The following notion is taken from~\cite{AM17}.
\begin{definition}[Symplectic growth]
	For a function $f:X\to\RR$, where $X=\NN$ or $X=\RR$, we define the exponential growth of $f$ as 
	\begin{equation*}
		\Gamma(f(a))=\limsup_{a\to\infty} \frac{1}{a}\log(f(a)).
	\end{equation*}
	If $\Gamma(f)>0$, we say that $f$ grows exponentially. We define the symplectic growth of the pair $(W,L)$ as the growth in dimension of the filtered wrapped Floer homology
	\begin{equation*}
		\Gamma^{\rm symp}(W,L)=\Gamma(\dim \WslantH^a(W,L)).
	\end{equation*}
\end{definition}
 
All but finitely many of the generators of the chain complexes underlying wrapped Floer homology correspond to Reeb chords from $\Lambda$ to itself, and the filtration corresponds to the length of these Reeb chords. With length we mean time of arrival.

Alves and Meiwes showed that if $\Lambda=\partial L$ is a sphere, then positivity of $\Gamma^{\rm symp}(W,L)$ implies that every Reeb flow on $M=\partial W$ has positive topological entropy. Our main result is that this theorem extends to positive contactomorphisms. 

\begin{theorem}\label{poscont}
	Under Assumption~\ref{f=0} assume that $\Gamma^{\rm symp}(W,L)>0$. Then the topological entropy of every positive contactomorphism of $(M,\xi)$ is positive.
\end{theorem}

Since a generic fiber of a cosphere bundle over a closed manifold satisfies Assumption~\ref{f=0}, this result also generalizes the works~\cite{MacSch11} mentioned earlier. 

Theorem~\ref{poscont} implies in particular that in the examples constructed by Alves and Meiwes every positive contactomorphism has positive topological entropy.
\begin{corollary}
	Let $M$ be the sphere $S^{2n+1}$ of dimension $2n+1\geq 7$, or $S^3\times S^2$, or the boundary of a plumbing tree whose vertices are unit codisc bundles over manifolds of dimension $\geq 4$. Then $M$ admits a contact structure $\xi$ such that every positive contactomorphism of $\xi$ has positive topological entropy.  
	%
	%
\end{corollary}

\subsubsection*{Method of proof} Alves and Meiwes prove their theorem using wrapped Floer homology $\WslantH$, which is a Lagrangian (or open string) version of symplectic homology. $\WslantH$ has the advantage that it admits product structures, notably a Pontrjagin product, and is functorial under various geometric operations, which Alves and Meiwes ingeniously combine to find examples such that $\WslantH(W,L)$ has exponential growth. Then they construct a $\WslantH(W,L)$-module structure on the wrapped Floer homology $\WslantH(W,L\to L')$, whose generators are Reeb chords from $L$ to nearby Lagrangians $L'$ to find positive volume growth and thus positive topological entropy.

In our time-dependent case we did not succeed to prove Theorem~\ref{poscont} by working with $\WslantH$ alone. Instead, we also work with time-dependent Lagrangian Rabinowitz--Floer homology (abbreviated by $\TH$), which is Lagrangian Rabinowitz--Floer homology based on a time-dependent Hamiltonian and whose generators correspond to time-dependent Reeb chords from $L$ to itself, see Section~\ref{sec:tlrfh} for the definition. We use $\TH$ because so far there seems to be no wrapped Floer homology that encodes time-dependent Reeb dynamics in a transparent way. The problem with $\WslantH$ is that the main tool for understanding the homology are radial Hamiltonians, whose radial coordinate explicitly corresponds to the slope and thus to the length of Reeb chord. But for time-dependent Hamiltonians this correspondence breaks down since time-dependent Hamiltonians are not constant along their chords. 

In $\TH$, however, the information of the length of a chord and its radial position are decoupled, thus the loss of radial control does not affect the understanding of the dynamics. On the other hand, it seems hard to set up a Pontrjagin product structure on $\TH$ or a $\TH(W,L)$-module structure on the homology $\TH(W,L\to L')$ whose generators correspond to time-dependent Reeb chords from $L$ to $L'$. Our solution is to combine the advantages of the two theories: We show that the growth of $\WslantH$, which was obtained through algebraic structures by Alves and Meiwes, implies growth of $\TH$, which then can be used to count the chords of a time-dependent Hamiltonian. 

The transition to counting chords between different Lagrangians $L, L'$ is then performed inside the action functional for $\TH$, capitalizing on the fact that in $\TH$ we can encode geometric information directly in the Hamiltonian. Thus, we do not need a module structure to deduce volume growth from growth of $\TH$.  

To relate $\WslantH$ to TH, we use various intermediate homologies, closely following~\cite{CFO10}. As a first step we use V-shaped wrapped Floer homology $\WcheckH$ (in the language of~\cite{CO17} the wrapped Floer homology of the trivial Lagrangian cobordism). An alternative approach would be to follow~\cite{CFO10} and to elaborate a long exact sequence connecting wrapped Floer homology, wrapped Floer cohomology and the Morse cohomology of Lagrangians and Legendrians. However, since we are only interested in the asymptotic behavior of the homology, it is enough to relate the positive parts of the homologies, which is shorter since we do not need to consider cohomology. 

\subsubsection*{Propositions along the way} The following is a list of our results that combine to the proof of Theorem~\ref{poscont}. The individual results might be of independent interest.

\begin{proposition}\label{acuteischeck}
	Under Assumption~\ref{f=0}, for all $a,b\notin\Scal$ with $0<a<b$ we have
	\begin{equation*}
		\WslantH^{(a,b)}(W,L)\cong\WcheckH^{(a,b)}(W,L),
	\end{equation*}
	where $\Scal$ is the set of lengths of Reeb chords from $L$ to $L$. These isomorphisms commute with morphisms induced by inclusion of filtered chain complexes.
\end{proposition}

The V-shaped wrapped Floer homology $\WcheckH$ can be identified with the standard Rabinowitz--Floer homology $\AH$ of $(W,L)$ (Here the A stands for autonomous) through Rabinowitz--Floer homologies with perturbed Lagrange multiplier. This identification is analogous to the long exact sequence connecting symplectic homology and closed string Rabinowitz--Floer homology discovered in~\cite{CFO10}:

\begin{proposition}\label{alrfhvwh}
	Under Assumption~\ref{f=0}, for all $a,b\notin\Scal$ with $-\infty<a<b<\infty$ we have
	\begin{equation*}
		\WcheckH^{(a,b)}(W,L)\cong \AH^{(a,b)}(W,L).
	\end{equation*}
	These isomorphisms commute with morphisms induced by inclusion of filtered chain complexes.
\end{proposition}
When combined, these two theorems imply that the positive part of $\WslantH$ coincides with the positive part of AH in a way that preserves the filtration.

%

The dimension of the positive part of AH is a lower bound to the number of Reeb chords from $\Lambda$ to $\Lambda$. We now deform the action functional of $\AH$ to the one of time-dependent Rabinowitz--Floer homology TH in order to count time-dependent Reeb chords from $\Lambda$ to $\Lambda$. We show that monotone deformations do not decrease the growth of the dimension of filtered homology groups, and by a sandwiching argument we then show\,:

\begin{proposition}[Preservation of positivity of growth]\label{continuation}
	Let $(W,L)$ and $h^t$ be as in Assumption~\ref{classofh} below, which is analogous to Assumption~\ref{f=0} in the new setup. 
	If the exponential dimensional growth of $\TH^{(0,T)}(h^t)$ is positive, then the exponential dimensional growth of $\TH^{(0,T)}(\tiilde h^t)$ is also positive for every other $\tiilde h^t$ that satisfies Assumption~\ref{classofh}.
	
	Quantitatively, if $c\leq h^t\leq C$, then the exponential dimensional growth $\gamma$ of $\TH^{(0,T)}(h^t)$ satisfies $c\Gamma^{\rm symp}(W,L)\leq \gamma\leq C\Gamma^{\rm symp}(W,L)$.
\end{proposition}	

Finally, to find positive topological entropy, we need to count the chords of our positive path of contactomorphisms between {\it different} Legendrians. The following proposition is established again by deforming the functional.

\begin{proposition}\label{changeLegendrian}
	Let $(W,L)$ and $h^t$ be as in Assumption~\ref{classofh}. Suppose that $\Gamma^{\rm symp}(W,L)>0$. Let $\Lambda'$ be a Legendrian that is isotopic through Legendrians to $\Lambda=\partial L$. Then the number of $\varphi^t$-chords from $\Lambda$ to $\Lambda'$ of length $\leq T$ grows exponentially. 
	
	Quantitatively, let $\psi$ be a contactomorphism that takes $\Lambda$ to $\Lambda'$ so that $(\psi^{-1})^*\alpha=f\alpha$. Then the exponential growth of the number of $\varphi^t$-chords from $\Lambda$ to $\Lambda'$ of length $\leq T$ is at least $\min f \cdot\min h^t\cdot\Gamma^{\rm symp}(W,L)$.
\end{proposition}

\subsection*{Organization of the paper} 
	In Section~\ref{sec:wfh} we describe the general geometric setup and the Floer homology of an action functional in a generality that suffices for this paper. In Subsections~\ref{sec:wh} and~\ref{sec:vwh} we describe wrapped Floer homology and V-shaped wrapped Floer homology and show that their positive parts coincide (Proposition~\ref{acuteischeck}). 
	
	In Section~\ref{sec:lrfh} we discuss the results concerning Rabinowitz--Floer homology. In Subsection~\ref{sec:alrfh} we define the standard autonomous Rabinowitz--Floer homology AH. In Subsection~\ref{sec:sequence} we show that AH is isomorphic to V-shaped wrapped Floer homology $\WcheckH$(Proposition~\ref{alrfhvwh}) by introducing perturbed Rabinowitz--Floer homology. In Subsection~\ref{sec:tlrfh} we introduce time-dependent Rabinowitz--Floer homology TH and show how the homological growth changes under the change of the dynamics (Propositions~\ref{continuation} and~\ref{changeLegendrian}).
	
	In Section~\ref{sec:proof} we puzzle together all these results to prove Theorem~\ref{poscont}.

\subsection*{Acknowledgments}
	I wish to thank Marcelo Alves, Matthias Meiwes, Peter Albers and Felix Schlenk for help and advice. This work is supported by SNF grant 200021-163419/1.

\section{Wrapped Floer Homology}\label{sec:wfh}

In this section we explain the theorems concerning wrapped Floer homologies. We begin with an exposition of the geometric setup, where we explain the terms in Assumption~\ref{f=0} and justify these assumptions. Then we outline the construction of Lagrangian Floer homology in a general setting. 

In the two following subsections we present wrapped Floer homology and V-shaped wrapped Floer homology and show that their positive parts coincide. In all versions we use $\ZZ_2$-coefficients and no grading. The wrapped Floer homology we present here was introduced in~\cite{AS10} and coincides with the version in~\cite{AM17}. For the entire section we follow~\cite{CFO10}, where analogous results for symplectic homology were established.

\subsubsection*{Liouville domains} A Liouville domain $(W,\omega,\lambda)$ is a compact manifold $W$ with boundary $\partial W=M$ endowed with an exact symplectic form $\omega=d\lambda$ and a choice of primitive $\lambda$ such that the so-called Liouville vector field $Y$ defined by $\iota_Y\omega=\lambda$ is transverse to the boundary, pointing outwards. Then $(\partial W=M,\xi=\ker\alpha)$, where $\alpha=\lambda|_M$, is a contact manifold. Let $\haat W=W\cup_M M\times[1,\infty)_r$ be the completion of $W$. The symplectization $(\haat M=M\times\RR^{>0},d(r\alpha))$ embeds into $\haat W$ such that $M\times\{1\}=M$, such that $\lambda=r\alpha$ and such that the Liouville vector field coincides with $r\partial_r$ on $\haat M$. 

\begin{example}\label{example}
	A starshaped domain $(D,dy\wedge dx,\frac12(y\,dx-x\,dy))$ in $\RR^{2n}$ is a Liouville domain with completion $\RR^{2n}$ and Liouville vector field $x\partial_x+y\partial_y$.

	Similarly, the sublevel $(D^*Q,dp\wedge dq, p\,dq)$ of a fiberwise starshaped hypersurface of $T^*Q$ is a Liouville domain with completion $T^*Q$ and Liouville vector field $p\,\partial_p$.
\end{example}

\subsubsection*{Asymptotically conical exact Lagrangians}
Let $L\subset W$ be a Lagrangian submanifold with Legendrian boundary $\partial L=\Lambda\subset M$. We say that $L$ is conical in a set $U\subset W$ if the Liouville vector field is tangent to $L\cap U$. We assume that
\begin{itemize}
\item $L$ is exact, i.e.\ $\lambda|L=df$ for some function $f:L\to \RR$,
\item $L$ is asymptotically conical, i.e.\ $L$ is conical in $M\times [1-\varepsilon,1]$ for $\varepsilon>0$ small enough.
\end{itemize}
An exact asymptotically conical Lagrangian satisfies $L\cap (M\times [1-\varepsilon,1])=\Lambda\times [1-\varepsilon,1]$ for $\varepsilon>0$ small enough. Since $\Lambda$ is Legendrian and $\lambda$ vanishes along $\partial_r$, $\lambda|_L$ vanishes in the region where $L$ coincides with $\Lambda\times [1-\varepsilon,1]$, and hence $f$ is locally constant in this region. Thus one can extend an asymptotically conical exact Lagrangian $L$ to an exact Lagrangian $\haat L=L\cup_\Lambda(\Lambda\times[1,\infty))$ in $\haat M$ by extending $f$ locally constantly. We will also refer to $\haat L$ as asymptotically conical.

Later, we are mainly interested in the case where $\Lambda$ is a sphere, so we will assume throughout that $\Lambda$ is connected. We can modify the Liouville domain such that $\lambda|_L=0$ as follows. For connected $\Lambda$ and with $\lambda|_L=df$, we can change $f$ by a constant such that $f\equiv 0$ on a collar neighborhood of $\haat W\backslash W$. Extend $f$ to a function $F$ on $\haat W$ with support inside $W$ minus a collar neighborhood of the boundary, and then add $-dF$ to $\lambda$. With respect to this new $\lambda$ the Lagrangian is still exact with $\lambda|_{L}\equiv 0$. This changes $\lambda$ in the interior of $W$, but not on the boundary $\partial W=M$, and thus also the Reeb flow on $M$, in which we are ultimately interested, is unchanged.

\begin{example}[Continuation of Example~\ref{example}]\label{exampl}
Any Lagrangian plane through 0 is an (asymptotically) conical Lagrangian in $(\RR^{2n},dy\wedge dx,\frac12(ydx-xdy))$.

A cotangent fiber $T^*_qQ$ is an (asymptotically) conical Lagrangian in $T^*Q$.
\end{example}

\begin{remark}
	If $L_1,L_2$ are two asymptotically conical exact Lagrangians that intersect, then we can in general not change $\lambda$ such that Assumption~\ref{f=0} holds for both Lagrangians simultaneously. As a result there is an additional term $+[f_1(x(0))-f_2(x(1))]$ in the action functional $\Acal_H$ defined below. Therefore the intersection of the Lagrangians, that later on will correspond to constant orbits, will not have zero action and thus one cannot separate the constant orbits from the others by action. The subsequent results should also be valid for a pair of Lagrangians, modulo finite dimensional terms stemming from the impossibility of separating different kinds of orbits. These terms do not influence the asymptotic behavior of the homology. In this paper, however, we take a different approach and consider pairs of Lagrangians only in the proof of Proposition~\ref{changeLegendrian}, where we use a trick to detect the chords between $L_1$ and $L_2$ in the space of paths from $L_1$ to $L_1$.
\end{remark}

\subsubsection*{Path space, Reeb chords and regularity}
For an asymptotically conical exact Lagrangian $L$ in $W$ we denote by $\Pcal(L)$ the space of smooth paths $x:[0,1]\to\haat W$ from $\haat L$ to $\haat L$ (i.e.\ $x(0),x(1)\in\haat L$). Denote by $R_\alpha$ the Reeb vector field of $\alpha$ on $M$ and by $\varphi_\alpha^t$ its flow. A {\it Reeb chord of length $T$} from $\Lambda$ to $\Lambda$ is a path $\gamma:[0,1]\to M$ such that $\dot\gamma=TR_\alpha$, where by length we mean the time it takes the Reeb flow to run through the cord. We call a Reeb chord of length $T$ {\it transverse} if the subspaces $T_{\gamma(1)}(\varphi_\alpha^T(\Lambda))$ and $T_{\gamma(1)}\Lambda$ of $T_{\gamma(1)}M$ intersect only in the origin. Note that the constant maps $t\mapsto x\in\Lambda$, which are Reeb chords of length 0, are never transverse. The spectrum of $(M,\alpha,\Lambda)$ is the set $\Scal(M,\alpha,\Lambda)$ ($\Scal$ for short) of lengths of Reeb chords from $\Lambda$ to $\Lambda$, including negative lengths for ``backward'' Reeb flows. This set is nowhere dense in $\RR$. 

Given a contact manifold $(M,\xi)$, the pair $(\alpha,\Lambda)$ consisting of a contact form $\alpha$ for $\xi$ and a Legendrian submanifold $\Lambda$ is called \textit{regular} if all nonconstant Reeb chords of $\alpha$ from $\Lambda$ to $\Lambda$ are transverse. Given a Liouville domain $(W,\omega,\lambda)$, the pair $(\lambda, L)$ consisting of the Liouville form and an asymptotically conical exact Lagrangian is called \textit{regular} if $(\lambda|_{M=\partial W},\Lambda=L\cap M)$ is regular.

\subsubsection*{Discussion of Assumption~\ref{f=0}}
Examples~\ref{example},~\ref{exampl} are natural examples of Liouville domains with asymptotically conical exact Lagrangians with spherical boundary. We have also seen that for $(W,L)$ with $\partial L$ connected, we can modify the Liouville form $\lambda$ such that $\lambda|_L=0$ without changing the Reeb dynamics of $\lambda|_{\partial W}$ on $\partial W$. 

If $\pi_1(L)=0$, then the assumption $[\omega]|_{\pi_2(W,L)}=0$ holds automatically since any disk $D$ with boundary on $L$ can be completed by a disk in $L$ to a sphere $S$ such that $\int_D\omega=\int_S\omega =\int_{\partial S}\lambda=0$ since $L$ is Lagrangian and $\omega$ is exact. If $\pi_1(L)\neq 0$, then this assumption is nontrivial. We make this assumption to prevent bubbling of holomorphic disks.

The strong assumption in \ref{f=0} is that $(\lambda,L)$ is regular. For given $L$ and generic $\lambda$ this is the case, so we can force regularity by perturbing the dynamics. If we choose not to perturb $\lambda$, then we have to face the fact that there are forms $\lambda$ that are not regular for any Lagrangian $L$. For example, the unit codisc bundle over the round sphere has periodic Reeb flow on its boundary, and thus any Legendrian gets mapped to itself after a full period, resulting in high degeneracy. Other degenerate examples are exact fillings of exactly fillable prequantization bundles, e.g.\ the subenergy level of the harmonic oscillator on $\RR^{2n}$ (the energy level $S^{2n-1}$ is a prequantization bundle over $\CC P^{n-1}$). Note that to apply Theorem~\ref{poscont}, it suffices to find positive symplectic growth for {\it one} regular pair $(\lambda,L)$. Thus, as long as one is able to guarantee positive symplectic growth, one is free to perturb $\lambda$. This is the case for the examples in~\cite{AM17}, where positive symplectic growth is guaranteed algebraically for a regular pair $(\lambda,L)$ that Alves and Meiwes construct from any regular pair $(\lambda',L')$, where $\lambda'$ restricts on the boundary to the standard contact structure and the constructed $\lambda$ restricts on the boundary to a dynamically exotic contact structure. 


\subsubsection*{Action functionals}
For a Hamiltonian $H:\haat W\to\RR$, the Hamiltonian vector field~$X_H$ is defined by $\iota_{X_H}\omega=\omega(\cdot,X_H)=dH$. Its flow is denoted by $\varphi_{H}^t$. We define the action functional $\Acal_H:\Pcal(L)\to\RR$ by
\begin{equation}
	\Acal_H(x) = \int_0^1 x^*\lambda-\int_0^1 H(x(t))\;dt.
\end{equation}
The critical points of $\Acal_H$ are Hamiltonian chords with $x(0),x(1)\in \haat L$. We denote the set of critical points by $\Crit\Acal_H$. A Hamiltonian $H$ is called {\it regular}\/ if $\haat L$ and $\varphi^1_H(\haat L)$ intersect transversely (i.e.\ all critical points of $\Acal_H$ are non-degenerate). We call the Hamiltonian {\it Morse--Bott regular}\/ if $\haat L$ and $\varphi^1_H(\haat L)$ intersect in closed manifolds such that $T(\haat L\cap \varphi^1_H(\haat L))=T\haat L\cap T\varphi^1_H(\haat L)$. Note that regular implies Morse--Bott regular. 

We will later specify the Hamiltonians we use, by imposing in particular a certain behavior at infinity. We assume throughout and without mentioning that all our Hamiltonians are regular, except if we consider Morse--Bott situations. Since the only non-regular behavior will happen at the set of constant orbits, and since we are ultimately interested in the asymptotic behavior of the homology, the Morse--Bott situation is of marginal interest and not elaborated here. For an exposition, see for example~\cite{CF09}.

\subsubsection*{Floer strips} 
An almost complex structure $J$ on $\haat W$ compatible with $\omega$ is called {\it conical}\/ at a point in $\haat M$ if it commutes with translations in the $r$-coordinate, preserves $\xi$ and sends the Reeb vector field to the Liouville vector field $J R_\alpha=r\partial_r$. Further, we call $J$ {\it asymptotically conical}\/ if $J$ is conical on $M\times[r,\infty)$ for some $r>0$. Using an asymptotically conical almost complex structure~$J$, we can define the $L^2$-metric on $\Pcal(L)$ by 
\begin{equation*}
	\langle \xi_1,\xi_2\rangle = \int_0^1\omega(\xi_1,J\xi_2)\; dt.
\end{equation*}
We interpret negative gradient flow lines $x_s(t)$ of $\Acal_H$ as Floer strips $u:\RR\times[0,1]\to\haat W$,
\begin{equation}\label{floer}
	\begin{cases}
		\partial_su+J(\partial_t u- X_H)=0,\\
		u(\cdot,i)\in\haat L,\; i=0,1.
	\end{cases}
\end{equation}
We switch between the notations $x_s(t)$ and $u(s,t)$ according to whether we wish to see this object as a negative gradient flow line or as a perturbed holomorphic curve. 

Given two critical points $x_+$ and $x_-$, we define the moduli space of parametrized Floer strips
\begin{equation*}
	\tiilde \Mcal(x_-,x_+,H,J)=\{\mbox{$x_s$ Floer strip, }\lim_{s\to\pm\infty}x_s=x_\pm\mbox{ uniformly in $t$}\}.
\end{equation*}
In the sequel we suppress $H$ and $J$ in the notation. Denote by $\tiilde \Mcal^k(x_-,x_+)$ the subset of $\tiilde \Mcal(x_-,x_+)$ on which the operator obtained by linearizing Floer's equation~(\ref{floer}) has Fredholm index $k$. 
There is an $\RR$-action on $\tiilde \Mcal(x_-,x_+)$ coming from translations on the domain in the $s$-variable. Denote the quotient by this action by
\begin{equation*}
	\Mcal^k(x_-,x_+)=\tiilde\Mcal^{k+1}(x_-,x_+)/\RR.
\end{equation*}
The {\it energy} of $u\in\tiilde\Mcal(x_-,x_+)$ is given by 
\begin{equation*}
	E(u):=\int_{-\infty}^\infty\langle\nabla\Acal_H(x_s),\nabla\Acal_H(x_s)\rangle\; ds=\Acal_H(x_-)-\Acal_H(x_+),
\end{equation*}
a quantity that is invariant under translation of the domain and thus descends to the quotient.
Since $E(u)$ is non-negative, $\Mcal(x_-,x_+)$ is empty if $\Acal_H(x_-)<\Acal_H(x_+)$.

From now on we assume that $\Mcal^k(x_-,x_+)$ is a $k$-dimensional manifold that is compact modulo breaking. Compactness modulo breaking follows from $L^{\infty}$- bounds on $u$ and its derivatives by bubbling analysis, and the manifold property follows if one can show that the set $\Mcal(x_-,x_+)$ is cut out transversally from the space of all smooth strips from $x_-$ to $x_+$ with boundary on $\haat L$. For a regular Hamiltonian with appropriate asymptotic behavior these two properties are satisfied for a generic asymptotically conical almost complex structure. For all the Floer homologies in this section these are classical facts.

For Morse--Bott regular Hamiltonians we consider moduli spaces of flow lines with cascades, where compactness modulo breaking and transversality hold for an additional generic choice of Riemannian metric on the critical manifolds. 

\subsubsection*{Floer chain complex and homology}

To define a homology, for $a\in\RR\backslash\Scal$ assume that the number of critical points of $\Acal_H$ with action less than $a$ is finite. Then we consider as chain group the free $\ZZ_2$-vector space
\begin{equation*}
	\FC^a(H,J,L)=\bigoplus_{x\in\Crit\Acal_H,\; \Acal_H(x)<a}\ZZ_2\cdot x.
\end{equation*}
We abbreviate $\FC^a:=\FC^a(H,J,L)$ and $\FC:=\FC^\infty:=\bigcup \FC^a$. We equip $\FC$ with a boundary operator $\partial:\FC\to \FC$ by counting isolated Floer strips mod $2$,
\begin{equation*}
	\partial x=\sum_{y\in\Crit\Acal_H}\#_{\ZZ_2}\Mcal^0(x,y)\cdot y.
\end{equation*}
There are only finitely many nonzero summands since Floer strips decrease in action and $\FC^a$ is finite. Every summand is well-defined since $\Mcal^0(x,y)$ is a compact 0-manifold and thus finite. The operator $\partial$ is therefore well defined. The property $\partial^2=0$ holds since broken flow lines from $x$ to $y$ via intermediate critical points form exactly the boundary $\partial\Mcal^1(x,y)$, as is seen by gluing and thus come in pairs. Hence, $(\FC,\partial)$ forms a chain complex, and we can define its homology $\FH=\ker\partial/\im\partial$.

Since Floer strips decrease in action, the boundary operator descends to a boundary operator $\partial^a$ on $\FC^a$. Further, we can define the chain complex with action window $(a,b)$, $a,b\notin\Scal$, as the quotient	$\FC^{(a,b)}=\FC^b/\FC^a$. This yields $\RR$-filtered Floer homology groups
\begin{align*}
	\FH^{a}&=\ker\partial^{a}/\im\partial^{a},\\
	\FH^{(a,b)}&=\ker\partial^{(a,b)}/\im\partial^{(a,b)}.
\end{align*}

As for all $\RR$-filtered homologies we have for $a<b<c$ long exact sequences
\begin{equation}\label{longexact}
	\ldots\to \FH^{(a,b)}\to \FH^{(a,c)}\to \FH^{(b,c)}\to \FH^{(a,b)}\to\ldots
\end{equation}
where the first two arrows are induced by inclusion of chain complexes.


In the following we investigate two classes of admissible Hamiltonians that will in a direct limit result in different versions of wrapped Floer homology. The first is the standard $\WslantH$, the second the V-shaped $\WcheckH$.

\subsection{Wrapped Floer homology}\label{sec:wh}
We begin with the classical wrapped Floer homology, as defined in~\cite{AS10}. 

\subsubsection*{Admissible Hamiltonians}
We say that a Hamiltonian $H:\haat W\to\RR$ is {\it $\WslantH$-admissible with slope $\mu>0$} if 
\begin{equation*}
	\begin{cases}
		H<0 \mbox{ on $W$,}	\\
		\exists b<-\mu\colon H(x,r)=h(r)=\mu r+b\mbox{ on $M\times[1,\infty)$.}
	\end{cases}
\label{acuteadmissible}
\end{equation*}
Denote by $\acute\Hcal$ the set of $\WslantH$-admissible regular Hamiltonians. For $H\in\acute\Hcal$, the slope~$\mu$ is not in $\Scal$, and so its orbits $x\in\Crit\Acal_H$ have image in $W$. 

If $H$ only depends on $r$ in $\haat W$ and is constant $<0$ for $r<1-\delta$, then $X_H=(\frac d{dr}H) R_{\alpha}$. Thus, the elements of $\Crit\Acal_H$ have constant $r$-coordinate and correspond to the Reeb chords in $M$ from $\Lambda$ to $\Lambda$ of period $\frac d{dr}H$ (they run backwards if $\frac d{dr}H<0$). Of course, such a Hamiltonian is not regular at points where $\frac d{dr}H=0$. This can be mended by adding a $C^2$-small Morse function supported in the region where $\frac d{dr}H$ is smaller than $\min\Scal$, which perturbs all the constant orbits and leaves the interesting Reeb orbits unchanged. 


\subsubsection*{Continuation morphisms and wrapped Floer homology}
A monotone increasing homotopy $H_s$ from $H_0$ to $H_1$ through admissible Hamiltonians induces a chain map $\FC(H_0)\to \FC(H_1)$ that decreases in action and thus restricts to a chain map $\FC^{(a,b)}(H_0)\to \FC^{(a,b)}(H_1)$ for $a,b\in\RR\cup\{\pm\infty\}\backslash\Scal$. The morphism 
\begin{equation*}\Phi_{H_s}:\FH^{(a,b)}(H_0)\to \FH^{(a,b)}(H_1)\end{equation*}
induced in homology is called {\it continuation morphism}. It is independent of the monotone homotopy $H_s$, and if $H_s$ does not depend on $s$, then $\Phi_{H_s}=id$. 

The set $\acute\Hcal$ admits a partial order where $H_0\leq H_1$ if the order is satisfied pointwise. With this partial order, $\acute\Hcal$ becomes a directed set. The set of homologies $\{\FH^{(a,b)}(H)\}$ thus forms a direct system indexed by $\acute\Hcal$. We define the wrapped Floer homology as the direct limit of this system, 
\begin{align*}
	\WslantH^{(a,b)}(W,L)=\lim_{\longrightarrow} \FH^{(a,b)}.
\end{align*}
For $a=-\infty$ we abbreviate $\WslantH^b(W,L):=\WslantH^{(-\infty,b)}(W,L)$. Since long exact sequences are preserved by direct limits, the sequence~(\ref{longexact}) holds for $\WslantH$. 

The generators of $\WslantH$ fall into two different classes: forward Reeb orbits on $M$, and ``short orbits'' on $W$. We are mainly interested in Reeb orbits of $(M,\lambda|_M)$. They are singled out by action, as the following lemma shows. 
\begin{lemma}\label{cut0off}
	In the geometric situation~(\ref{f=0}) and for positive $a\notin\Scal$ the positive part of the homology $\WslantH_+^{a}(W,L):=\WslantH^{(\varepsilon, a)}(W,L)$, where $0<\varepsilon<\min\Scal_{>0}$, is generated by the Reeb chords from $\Lambda$ to $\Lambda$ of length $<a$, and their action is given by their length.
\end{lemma}
\begin{proof}
	We start with Hamiltonians $H=H_{\mu,\varepsilon'}$ as in Figure~\ref{fig:Hmu} defined for $\mu\notin\Scal,$ $0<4\varepsilon'<\varepsilon$ by 
\begin{itemize}
	\item $H$ only depends on $r$,
	\item $H\equiv -\varepsilon'$ for $r<1$,
	\item $H=\mu(r-1)-2\varepsilon'$ for $r\geq1+\varepsilon'$,
	\item $H$ is convex.
\end{itemize}


\begin{figure}[h]
	\centering
		\includegraphics{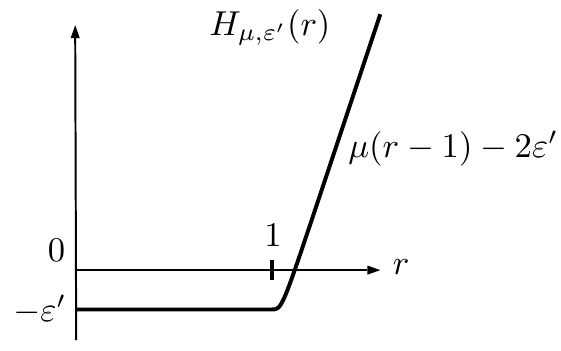}
	\caption{The function $H_{\mu,\varepsilon'}$}
	\label{fig:Hmu}
\end{figure}

\noindent Then we perturb $H$ in the region where $|\frac{d}{dr}H(r)|<\varepsilon'$ to a function still denoted by~$H$ such that $L$ and $\varphi_H^1(L)$ intersect transversely and such that in the perturbation region the new $H$ satisfies $\|H\|_{C^0}\leq \varepsilon'$ and $\|\lambda(X_H)\|_{C^0}<\varepsilon'$, making it a regular admissible Hamiltonian in $\acute\Hcal$. Note that $\{H_{\mu,\varepsilon'}\}$ is cofinal in $\acute\Hcal$ for $\mu\to\infty,$ $\varepsilon'\to0$. Then critical points $x$ of $\Acal_H$ correspond either to Reeb chords from $\Lambda$ to $\Lambda$ with $H(x(t))$ and $(\frac d{dr}H)(x(t))$ constant in $t$ and have action 
	\begin{eqnarray*}
		\Acal_H(x) &=& \int_0^1 x^*\lambda-\int_0^1 H(x(t))\;dt\\
		&=& \int_0^1\lambda\left(\frac d{dr}H \cdot R_{\lambda|_M}(x(t))\right)\; dt-H\\
		&=& \frac d{dr}H-H,
	\end{eqnarray*}
or they are $\varepsilon'$-short $X_H$-chords, namely $\int_0^1|\lambda(\dot x(t))|\;dt<\varepsilon'$, from $L$ to $L$, such that 
	\begin{eqnarray*}
		|\Acal_H(x)| &\leq& \int_0^1 |\lambda(\dot x(t))|+ |H(x(t))|\;dt\leq 2\varepsilon'\\
		&\leq& \frac12\varepsilon.
	\end{eqnarray*}
In the first case the term $H(x)$ tends to zero in the direct limit, so Reeb chords have limit action $\frac d{dr}H={\rm length}(x)>\varepsilon$, and in the second case the action of the critical points lies outside the action window.  
\end{proof}

In analogy to the positive part we define the non-positive part $\WslantH^0(W,L):=\WslantH^{\varepsilon}(W,L)$ for $0<\varepsilon<\min\Scal_{>0}$. For all $a\notin\Scal$ the long exact sequence~(\ref{longexact}) becomes
\begin{equation*}
	\ldots\to \WslantH^0(W,L)\to \WslantH^{a}(W,L)\to \WslantH_+^{a}(W,L)\to \WslantH^0(W,L)\to\ldots.
\end{equation*}
One can perform the perturbation of the family of functions $\{H_{\mu,\varepsilon'}\}$ in the proof above such that finite set $L\cap\varphi_H^1(L)$ in the perturbation region is constant, which implies that $\WslantH^{a}(W,L)$ and $\WslantH^{a}_+(W,L)$ are isomorphic up to an error of finite dimension independent of $a$. Thus, $\WslantH$ grows exponentially if and only if $\WslantH_+$ grows exponentially. In~\cite{CO17} it is mentioned that $\WslantH^0(W,L)$ corresponds to the Morse-cohomology of $L$. 

Note that even though $\WslantH$ is defined as a direct limit, for finite action windows $(a,b)$ the homology $\WslantH^{(a,b)}$ is already attained by a Hamiltonian $H\in\acute \Hcal$ that is $C^2$-small for $r<1$, at $r=1$ sharply increases and has asymptotic slope $\mu>b$.

\subsection{V-shaped wrapped Floer homology}\label{sec:vwh}

We construct V-shaped wrapped Floer homology by using a different class of Hamiltonians. A Hamiltonian is called $\WcheckH$-admissible if
\begin{equation*}
	\begin{cases}
		H<0 \mbox{ on $M\times\{1\}$,}	\\
		\exists b<-\mu\colon H(x,r)=h(r)=\mu r+b\mbox{ on $M\times[1,\infty)$.}
	\end{cases}
\label{vadmissible}
\end{equation*}
Denote the set of $\WcheckH$-admissible regular Hamiltonians by $\widecheck \Hcal$. Again, using continuation homomorphisms we can define for $a,b\notin\Scal$ the direct limit homology 
\begin{equation*}
	\WcheckH^{(a,b)}(W,L)=\lim_{\longrightarrow} \FH^{(a,b)}(W,L).
\end{equation*}
In the language of~\cite{CO17} this is the homology of the trivial Liouville cobordism with Lagrangian $([0,1]\times M, [0,1]\times \Lambda)$ with filling $(W,L)$. This homology is different from wrapped Floer homology. The paper~\cite{CFO10} suggests that there is a long exact sequence splitting $\WcheckH$ into wrapped Floer homology and wrapped Floer cohomology with interesting behavior in the ``$0$-part'' $\WcheckH^{(-\varepsilon,\varepsilon)}(W,L)$. Since we are only interested in the positive part of the homology, Proposition~\ref{acuteischeck} is sufficient for our purposes.

\begin{proof}[Proof of Proposition~\ref{acuteischeck}]
We proceed by deforming the Hamiltonians. We consider a cofinal family of Hamiltonians in $\acute \Hcal$ and show that each such Hamiltonian can be deformed to a Hamiltonian in $\widecheck\Hcal$ such that the set of deformed Hamiltonians forms a cofinal family. 
The cofinal family in $\acute\Hcal$ is the family $\{H_{\mu,\varepsilon'}\}$ from the proof of Lemma~\ref{cut0off}. Recall that critical points of $H_{\mu,\varepsilon'}$ are either $\varepsilon'$-short trajectories with action $\leq2\varepsilon'\leq\frac12\varepsilon$ or Reeb trajectories with length $>\varepsilon$ and action $>\varepsilon-\varepsilon'>\frac12\varepsilon$. Thus in the chain complex $\FC^{(\frac12\varepsilon,a)}(H_{\mu,\varepsilon'})$ the trajectories of the first type are quotiented out. 

To define the cofinal family in $\widecheck\Hcal$ choose $\delta>0$ such that $L$ is conical for $r\in[1-\delta,1]$ and $0<\varepsilon<\min|\Scal|$ such that $2\varepsilon<\delta$. Start with $G=G_{\mu,\nu,\varepsilon'}$ as depicted in Figure~\ref{fig:Gmunu} for $\nu\leq0<\mu$ with $\mu,\nu\notin\Scal, 0<4\varepsilon'<\varepsilon$ with the following properties:
\begin{itemize}
	\item $G$ depends only on $r$,
	\item $G(r)\equiv -\frac12\delta\nu$ for $r<1-\delta$,
	\item $G(r)= \nu (r-1)-2\varepsilon'$ for $r\in [1-\frac12\delta,1-\varepsilon']$,
	\item $G(1)= -\varepsilon'$ and $G'(0)=0$,
	\item $G(r)= \mu (r-1)-2\varepsilon'$ for $r\geq1+\varepsilon'$,
	\item $G$ is convex for $r\in[1-\varepsilon',1+\varepsilon']$ and concave for $r\in[1-\delta,1-\frac12\delta]$.
\end{itemize}

\begin{figure}[h]
	\centering
		\includegraphics{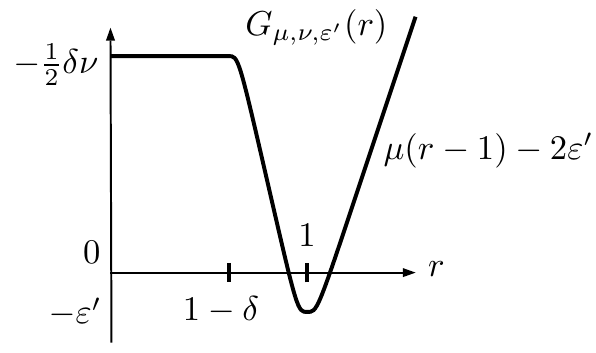}
	\caption{The function $G_{\mu,\nu,\varepsilon'}$}
	\label{fig:Gmunu}
\end{figure}

\noindent Then we perturb $G$ in the region where $|\frac{d}{dr}G(r)|<\varepsilon'$ to a function still denoted by~$G$ such that~$L$ and $\varphi_G^1(L)$ intersect transversely and such that in the perturbation region the new $G$ satisfies $\|G\|_{C^0}\leq \varepsilon'$ and $\|\lambda(X_G)\|_{C^0}\leq\varepsilon'$, making it a regular admissible Hamiltonian in $\acute\Hcal$. Note that for $\nu=0$ we have $G_{0,\mu,\varepsilon'}=H_{\mu,\varepsilon'}$. Critical points of $\Acal_G$ come in several types distinguished by their location. We can compute their action as in the proof of Lemma~\ref{cut0off}:
\begin{itemize}
	\item[(I)] $r\leq1-\delta$: $\frac14\varepsilon$-short trajectories with $G\sim-\frac12\delta\nu>0$ and with action $<\frac14\varepsilon-(-\frac12\delta\nu)<\frac12\varepsilon$,
	\item[(II)] $1-\delta<r<1-\frac12\delta$: backward Reeb trajectories with length in $(\nu,-\varepsilon)$, with $G\sim-\frac12\delta\nu>0$ and action $<-\varepsilon-(-\frac12\delta\nu)<0$,
	\item[(III)] $1-\varepsilon'<r<1$: backwards Reeb trajectories with length in $(\nu,-\varepsilon)$, with $G\sim-\varepsilon'$ and action $<-\varepsilon-(-\varepsilon')<0$,
	\item[(IV)] $r\approx1$: $\varepsilon'$-short trajectories with $G\sim-\varepsilon'$ and with action $<\varepsilon'-(-\varepsilon')<\frac12\varepsilon$,
	\item[(V)] $r>1$: Reeb trajectories with length in $(\varepsilon,\mu)$, $G\sim-\varepsilon'$ and action $>\varepsilon-(-\varepsilon')>\frac12\varepsilon$.
\end{itemize}
Thus in the chain complex $\FC^{(\frac12\varepsilon,a)}(G)$ all critical points other than of type (V) get quotiented out. It is clear from this description that we can monotonously deform $H_{\mu,\varepsilon'}=G_{0,\mu,\varepsilon'}$ to $G_{\nu,\mu,\varepsilon'}$ through Hamiltonians of type $G$ by lowering the parameter $\nu$. Since for $r\geq1$ this deformation does not change the function, we conclude that the continuation homomorphism $\FH^{(\frac12\varepsilon)}(H_{\mu,\varepsilon'})\to\FH^{(\frac12,a)}(G_{\nu,\mu,\varepsilon'})$ is lower diagonal and thus an isomorphism. Taking the limit $(\varepsilon',\nu,\mu)\to(-\infty,\infty,0)$ for $\FH^{(\frac12\varepsilon,a)}(G_{\nu,\mu,\varepsilon'})$  is thus the same as taking the limit $(\mu,\varepsilon')\to(0,\infty)$ for $\FH^{(\frac12\varepsilon,a)}(H_{\mu,\varepsilon'})$.

To show that the isomorphisms commute with morphisms induced by inclusion of filtered chain complexes note that for parameters $(\nu,\mu,\varepsilon')$ we have that the two chain complexes are not only isomorphic, but identical, $\FC^I(G_{(\nu,\mu,\varepsilon')})=\FC^I(H_{(\mu,\varepsilon')})$, for any interval $I=(a,b)$ with $\frac12\varepsilon<a<b\leq \infty$. This means that at the chain level inclusions trivially commute with the identity, thus morphisms induced by inclusion commute with isomorphisms induced by identity. Taking the direct limit preserves commutative diagrams and we are done.

For finite action windows one can even take a shortcut in the above argument since one can find parameters $(\nu,\mu,\varepsilon')$ sufficiently close to $(-\infty,\infty,0)$ such that $\FH^I(G_{\nu,\mu,\varepsilon'})\sim\WcheckH^I(W,L)$ and $\FH^I(H_{\mu,\varepsilon'})\sim\WslantH^I(W,L)$ for both $I=(a,b)$ and $I=(a',b')$ and the proof finishes before taking direct limits.
\end{proof}

As an alternative we can choose not to perturb $G$ around $r=1$. If $G'(0)=0$ and $G''(0)>0$, the critical manifold of type (IV) consists of constant orbits, can be identified with $\{0\}\times\Lambda$ and is Morse--Bott. This way it becomes transparent that the 0-part $\WcheckH^0(W,L)$ of V-shaped wrapped Floer homology can be identified with the Morse cohomology of $\Lambda$. 

\section{Lagrangian Rabinowitz--Floer Homology}\label{sec:lrfh}

We introduce three types of Lagrangian Rabinowitz--Floer homology. We start with an exposition of Lagrangian Rabinowitz--Floer homology with autonomous Hamiltonian (AH). This is the standard Lagrangian Rabinowitz--Floer homology. That the Hamiltonian is autonomous means that the critical orbits of the functional are contained in a fixed energy surface which leads to a much lighter analysis than for time-dependent Hamiltonians. In our construction of $\AH$ we will work with just one fixed Hamiltonian function $H$. While many different choices of autonomous Hamiltonians would result in isometric homologies $\AH$, where filtered versions have the same dimension growth, we do not elaborate on this, since this independence will later on automatically follow in the setting of $\TH$, where an even larger class of (time-dependent) Hamiltonians is used.

To show that AH is isomorphic to $\WcheckH$, we introduce Lagrangian Rabinowitz--Floer homology with perturbed Lagrange multiplier (PH), following~\cite{CFO10}. Since the proof of Proposition~\ref{alrfhvwh} can be found in the paper~\cite{CFO10} up to small changes, we only sketch the construction of $\PH$. 

Finally, we introduce Lagrangian Rabinowitz--Floer homology for time-dependent Hamiltonians (TH). We will study invariance of growth of TH under monotone changes of the Hamiltonian and derive uniform growth properties by a sandwich argument. Further, we show how to encode changes of the target Legendrian in the functional and how to derive uniform growth properties for time-dependent Reeb chords from $\Lambda$ to $\Lambda'$, where $\Lambda'$ is a Legendrian isotopic to $\Lambda$.

\subsection{Autonomous Lagrangian Rabinowitz--Floer homology (AH)}\label{sec:alrfh}
\subsubsection*{The action functional}
Let $H:\haat W\to \RR$ be a smooth function on $\haat W$ such that $0$ is a regular value (later $H$ is specifically chosen). We define the action functional $\acal_{H}:\Pcal(L)\times \RR\to\RR$ by 
\begin{eqnarray}\label{acal}
	\acal_{H}(x,\eta)&=&\int_0^1x^*\lambda-\eta\int_0^1H(x(t))\;dt.
\end{eqnarray}
A pair $(x,\eta)$ is a critical point of $\acal_H$ if and only if it satisfies the equations
\begin{equation*}\label{crit:alrfh}
\left\{\begin{array}{rcl}
		\dot x(t)&=&\eta X_{H}(x(t)),\\
		H\circ x&\equiv&0.
	\end{array}\right.
\end{equation*}
The first equation implies that $x$ is a Hamiltonian orbit from $L$ to $L$ with period $\eta$ (flowing backwards if $\eta<0$). The second equation implies that the image of $x$ is contained in the hypersurface defined by $H=0$.

If $H$ depends only on $r$, $H(r)=0$ only for $r=1$ and $H'(1)=1$, then $\Crit\acal_H$ is the set of Reeb orbits from $\Lambda$ to $\Lambda$ with period $\eta$ (running backwards if $\eta<0$), and $\acal_H(x,\eta)=\eta$ at critical points. Note that for $\eta=0$ the critical points are constant orbits that form the critical manifold $\Lambda=L\cap H^{-1}(0)$. Thus $\acal$ is never Morse. If $(W,L)$ is regular, all critical points with $\eta\neq0$ are regular and the critical manifold at $\eta=0$ is Morse--Bott. Since we only have one nontrivial critical manifold, we do not focus on the Morse--Bott situation. We choose a Morse function on $\Lambda$ and abusing notation we denote by $\Crit\acal_H$ the union of the isolated critical points of~$\acal_H$ with the critical points of the Morse function. The action of a critical point of the Morse function is, by definition, $0=\acal_H(\Lambda)$.

\subsubsection*{Choice of Hamiltonian}
We fix a smooth Hamiltonian such that
\begin{itemize}
	\item $H$ depends only on $r$,
	\item $H\equiv -\frac23\delta$ for $r<1-\delta$,
	\item $H=r-1$ for $r>1$,
	\item $H$ is convex.
\end{itemize}
Then the critical points of $\acal_H$ are as described above. The choices of constants are for compatibility with the other homologies.

\subsubsection*{Moduli spaces of Floer strips}
We choose again an asymptotically conical almost complex structure $J$ on $\haat W$. It induces the metric on $\Pcal(L)\times\RR$
\begin{equation*}
	\langle(\hat x_1,\hat\eta_1),(\hat x_2,\hat\eta_2)\rangle=\int_0^1\omega(\hat x_1,J\hat x_2)\; dt+\hat\eta_1\hat\eta_2.
\end{equation*}
For this inner product the $L^2$-gradient equation of $\acal_H$ is the Rabinowitz--Floer equation
\begin{equation*}
	\begin{cases}
		\partial_s x+J(x)[\partial_t x-\eta X_H(x(s,t))]=0,\\
		\partial_s\eta +\int_0^1H(x(s,t))\;dt=0.
	\end{cases}
\end{equation*}
In addition we choose a Riemannian metric $g$ on the only nontrivial critical manifold $\Lambda\times\{\eta=0\}$. For two critical points $(x_1,\eta_1)$ and $(x_2,\eta_2)$ of $\acal_H$ we consider the moduli space $\tiilde\Mcal((x_1,\eta_1),(x_2,\eta_2),H,J)$ of gradient flow lines with cascades from $(x_1,\eta_1)$ to $(x_2,\eta_2)$. We denote the subset where the linearization of the Floer equation has Fredholm index $k$ by $\tiilde\Mcal^k((x_1,\eta_1),(x_2,\eta_2),H,J)$. On this space there is a natural action by $s$-translation of the domain (if there are multiple cascades, then there is one such action per cascade). We take the quotient by this action to obtain the reduced moduli space of gradient flow lines $\Mcal^{k-1}((x_1,\eta_1),(x_2,\eta_2),H,J)$ (if there are multiple cascades, then the dimensional shift is by the number of cascades). 

For regular $(W,L)$ and generic $J$ and $g$ we have transversality and compactness modulo breaking for these moduli-spaces. 

\subsubsection*{Chain complex}
We define the chain groups $\AC^a(H)$ as free $\ZZ_2$-module generated over $\Crit\acal_H$,
\begin{equation*}
	\AC^a(H)=\sum_{\acal_H(x,\eta)<a}\ZZ_2\cdot(x,\eta).
\end{equation*}
The differential is defined by counting modulo $\ZZ_2$ isolated Rabinowitz--Floer-strips:
\begin{equation*}
	\partial(x,\eta)=\sum_{(x',\eta')\in\Crit\acal}\#_{\ZZ_2}\Mcal^0((x,\eta),(x',\eta'),H,J).
\end{equation*}
By a gluing argument we can identify $\partial^2$ with counting broken cascades in $\partial\Mcal^1$, which is zero modulo $2$. Thus we can define $\AH^{(a,b)}(W,L)$ as the filtered homology of this chain complex. 

\subsection{The relation between AH and V-shaped wrapped Floer homology}\label{sec:sequence}
The goal of this subsection is to prove Proposition~\ref{alrfhvwh}. It follows almost exactly as its analogue for the closed string case in~\cite{CFO10}. The difference is in the analysis when we show $L^\infty$-bounds for $x$ and its derivatives in order to prove compactness modulo breaking of moduli spaces of Floer strips $(x(s,t),\eta(s))$. For the $L^\infty$-bound on $x$ we invoke the maximum principle which is possible because Floer strips satisfy Neumann conditions at their boundary. For the $L^\infty$-bound on the first derivatives of $x$ we perform a bubbling analysis, where we use the hypothesis that $[\omega]$ vanishes on $\pi_2(W,L)$ to exclude bubbling of disks at the boundary. 

The remaining argumentation remains completely unchanged. We sketch it here, for details see~\cite{CFO10}. 

\subsubsection*{Perturbed Lagrangian Rabinowitz--Floer homology}
As our main tool for the proof we introduce perturbed Lagrangian Rabinowitz--Floer homology (PH), which is defined like AH but for the functional $\pacal:\Pcal(L)\times\RR\to\RR$,
\begin{equation*}
\pacal_{H,\alpha,\beta}(x,\eta)=\int_0^1 x^*\lambda - \alpha(\eta)\int_0^1H(x(t))\;dt+\beta(\eta),
\label{PA}
\end{equation*}
depending on the smooth function $H:\haat W\to\RR$ (with properties specified later on) and $\alpha,\beta:\RR\to\RR$. In this subsection we always assume $H=H(r)$ to depend only on~$r$. 

\subsubsection*{Characterization of the critical points}
A pair $(x,\eta)$ is a critical point of $\pacal_{H,\alpha,\beta}$ iff
\begin{equation}
\label{Pcrit}
	\begin{cases}
		\dot x(t)=\alpha(\eta)X_H(x(t)),\\
		\dot\alpha(\eta)\int_0^1 H(x(t))\;dt=\dot\beta(\eta).
	\end{cases}
\end{equation}
On the Morse--Bott component of the set of critical points, we consider critical points of a Morse function on this manifold. Note that, since~$H$ only depends on~$r$, $X_H=H'(r)R_\lambda$ and thus all critical points of $\pacal$ correspond to Reeb chords of length $T=\alpha(\eta)H'$. Passing from the set of Reeb chords of length $T$ to the corresponding set of critical points amounts to finding numbers $r,\eta$ such that 
\begin{equation}
\label{Pcrit'}
	\begin{cases}
		T=\alpha(\eta)H'(r),\\
		\dot\alpha(\eta)H(r)=\dot\beta(\eta).
	\end{cases}
\end{equation}
The action at a critical point equals $\pacal_{H,\alpha,\beta}(x,\eta)=\alpha(\eta)(H'(r) - H(r))+\beta(\eta)$.

\subsubsection*{Floer equations}
Given an asymptotically conical almost complex structure $J$, the negative $L^2$-gradient equation is equivalent to the Floer equations for $(u(s,t),\eta(s))$:
\begin{equation}
\label{Pfloer}
	\begin{cases}
		\partial_s u+J(u)[\partial_t u-\alpha(\eta)X_H(u)]=0,\\
		\partial_s\eta -\dot\alpha(\eta)\int_0^1H(u)\;dt+\dot\beta(\eta)=0.
	\end{cases}
\end{equation}
In the case of nontrivial critical manifolds we use an additional Riemannian metric to define negative gradient equations with cascades.

Since we use PH to interpolate between AH and $\WcheckH$, it is important to observe that the two homologies are indeed special cases of PH.

\subsubsection*{Special case 1: Rabinowitz--Floer homology}
Note that for $H$ as in AH and for $\alpha(\eta)=\eta,\beta(\eta)=0$ we have $\pacal_{H,\eta,0}=\acal_H$. Thus we see directly that 
$$\PH^{(a,b)}(H,\eta,0)\cong\AH^{(a,b)}(W,L).$$

\subsubsection*{Special case 2: Wrapped Floer homology}
Recall that even if $\WcheckH$ is defined as a limit, for finite action windows $(a,b)$ the limit is attained for $H\in\widecheck\Hcal$ that is $C^2$-small for $r<1$, then steeply increases and has slope at infinity $\mu>b$, and then $\FH^{(a,b)}(\Acal_H)=\WcheckH^{(a,b)}$. For $(H,\alpha,\beta)$ with such a function $H$, with $\alpha(\eta)=1$ and $\beta$ a Morse function with only one critical point in 0 the critical equation~(\ref{Pcrit}) splits: $(x,\eta)$ is a critical point iff $x$ is a critical point of $\Acal_H$ and $\eta$ is 0. The Floer equation~(\ref{Pfloer}) also splits: $(x(s,t),\eta(s))$ is a Floer trajectory if $x(s,t)$ is a Floer trajectory of $\Acal_H$ and $\eta(s)$ is a Morse gradient trajectory. Since the Morse complex of $\eta$ consists of just one point, we have directly 
$$\PH^{(a,b)}(H,1,\beta)\equiv \WcheckH^{(a,b)}(W,L).$$

\subsubsection*{Invariance property}
To prove Proposition~\ref{alrfhvwh} we thus need a way to show that different PH are isomorphic. Of course $\PH^{(a,b)}(H,\alpha,\beta)$ depends on the functions $(H,\alpha,\beta)$, but the following proposition shows that it is invariant under homotopies for which the spectrum does not cross the boundaries of the action window.

\begin{proposition}\label{Pinvariance} 
Let $(H_s,\alpha_s,\beta_s),s\in[0,1]$ be a homotopy that is supported in $(0,1)$, such that for all values of $s$ the resulting homology is well defined and such that for no value of $s$ the boundaries of the action window $(a,b)$ lie in the spectrum $\Scal_s$ of $\pacal_{H_s,\alpha_s,\beta_s}$. Then 
	\begin{equation*}
		\PH^{(a,b)}(H_0,\alpha_0,\beta_0)\cong \PH^{(a,b)}(H_1,\alpha_1,\beta_1).
	\end{equation*}
\end{proposition}
This result follows by the usual continuation technique. 

\subsubsection*{The steps connecting the special cases}
One can connect the two special cases above by the five steps from~\cite{CFO10} such that no step has an action crossing and thus Proposition~\ref{Pinvariance} is applicable. We outline the steps (including a preparatory step), discuss what they do to the functional and give details where they differ from~\cite{CFO10}. Remember that we start with $(H,\eta,0)$, where $\pacal_{(H,\eta,0)}=\acal_H$ and want to end with $(H,1,\beta)$, where $\PH^{(a,b)}(H,1,\beta)=\WcheckH^{(a,b)}(W,L)$. We suppose that $a,b$ are not in the spectrum $\Scal$. Since the set $\Scal$ is nowhere dense, there is an $\varepsilon>0$ such that $a,b$ are $\varepsilon$-far from $\Scal$. 

We deform in the following steps, cf.\ Figure~\ref{fig:deformations}:

\begin{figure}[h]
	\centering{
		\includegraphics{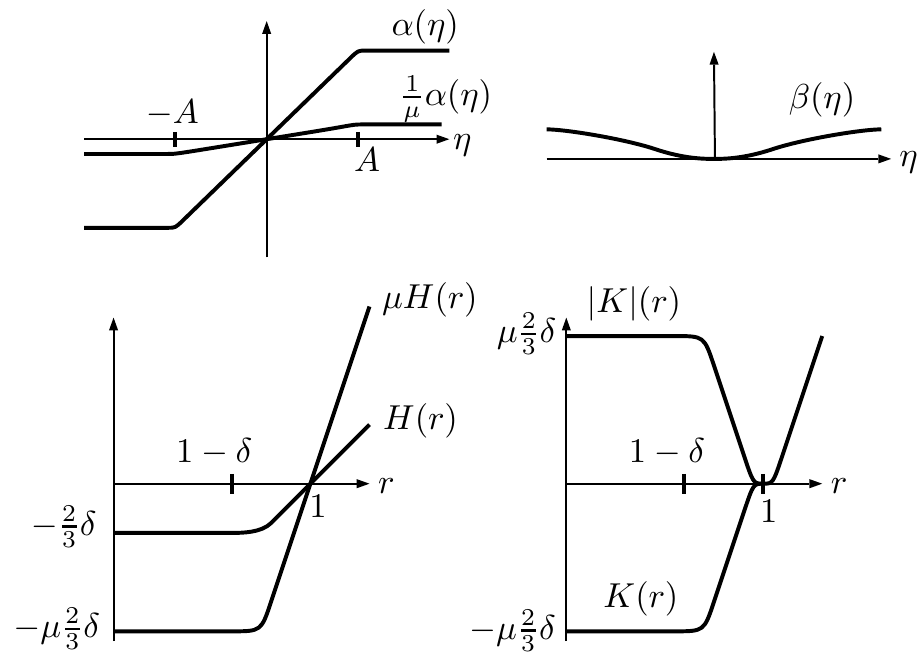}
	}
	\caption{The functions appearing along the deformation from $(H,\eta,0)$ to $(|K|,1,\beta)$}
	\label{fig:deformations}
\end{figure}

\begin{enumerate}
	\item Replace $(H,\eta,0)$ by $(H,\alpha(\eta),0)$ where $\alpha=id$ for $\eta\in(-A,A)$ and $\alpha=\pm A$ for $\pm\eta>A$ (up to a smoothing) for $A$ so large that no Floer strip $(x_s(t),\eta_s)$ with asymptotics in the chosen action window exceeds $|\eta_s|>A$, and such that $-\frac12\delta A<a$.
	\item Replace $(H,\alpha,0)$ by $(\mu H,\frac 1\mu \alpha,0)$ with $\mu=A$.
	\item Deform $(\mu H,\frac1\mu \alpha,0)$ to $(\mu H,\frac1\mu \alpha,\beta)$ for $\beta$ a $C^2$-small Morse function with unique minimum $\beta(0)=0$.
	\item Replace $(\mu H,\frac1\mu \alpha,\beta)$ by $(K,\frac1\mu \alpha,\beta)$ where $K$ is $\mu H$ with a small terrace point at $r=1$.
	\item Homotope $\pacal_{(K,\frac1\mu \alpha,\beta)}$ to $\pacal_{(|K|,1,\beta)}$. 
	\item Perturb $(|K|,1,\beta)$ to $\widecheck H\in\widecheck\Hcal$.  
\end{enumerate}

\subsubsection*{Discussion of the steps}
For each situation and the intermediate deformations we have to check that the triples $(H,\alpha,\beta)$ define a Morse--Bott action functional in the desired action window. To this end we have to show that the moduli spaces of Floer strips are compact modulo breaking which requires establishing $L^\infty$-bounds on $x$, $\eta$ and the derivatives of $x$ along Floer-strips. 

First we outline the proof that there is no action crossing, which coincides with seeing what effect the deformation of the functional has to the set of critical points and the action spectrum. For the complete discussion we refer the reader to~\cite[Section 6]{CFO10}.

\textbf{Step 1} reduces the support of $\dot\alpha$ to a compact interval. This step is a replacement, not a deformation. The Kazdan--Warner type inequality in~\cite{CFO10} suggests that we could also achieve this step through deformations, but it is then much harder to establish bounds (since $\dot\alpha$ is not compactly supported). The isomorphism can be shown directly: For a Reeb chord the equations~(\ref{Pcrit'}) have two kinds of solutions, the ones that coincide for $(H,\eta,0)$ and $(H,\alpha,0)$ and others that lie outside the action window in question. By the choice of $A$, also the Floer strips coincide and thus the chain complexes are the same, for a detailed discussion see~\cite[Section 5.4]{CFO10}. The reason for the condition $-\frac12\delta A<a$ is that it guarantees that critical points of $\Acal_H$ at $r<1-\frac12\delta$ have action outside the action window. These critical points will also be visible as critical points of $\pacal$ from Step 3 on, but not within the action window. 

\textbf{Step 2} gives $H$ the asymptotic slope $\mu=A$, and in compensation flattens $\alpha$ so that it is constant $\pm1$ for $|\eta|$ large. This leaves the action functional unchanged, hence induces trivially an isomorphism by the identity at the chain level.

\textbf{Step 3} puts $\beta$ into its intended form. One can choose the isotopy $(\mu H,\frac 1\mu\alpha,\beta_s)$ such that $\beta_s$ is always a $C^2$-small Morse function with one unique minimum $\beta_s(0)=0$ for $s>0$. The first equation of~(\ref{Pcrit'}) suggests that $\alpha(\eta)$ and thus $\eta$ is not changed by this and the second equation, that locally looks like $(r-1)=\dot\beta(\eta)$, has solutions for $r$ close to $1$ since $\beta_s$ is small. Thus the change moves nonconstant critical points of the functional slightly away from $\{r=1\}$, for $\eta>0$ to $r>1$ and for $\eta<0$ to $r<1$, which changes the action a little. Also, some new critical points appear around $r\sim 1-\delta$ for $\dot\alpha$ small enough to satisfy the second equation and thus $\eta\sim -A$. There $\pacal\sim -(\mu H'- (-\mu\delta)))<-\mu\delta<a$ and thus the critical points lie outside the action window. 

\textbf{Step 4} deforms $H$ to $K$ by introducing a terrace point in a neighborhood of $r=1$ that contains only constant $(H,\alpha,\beta)$-critical points. The deformation to the terrace translates the existing nonconstant critical points in the $r$-direction by a small amount and leaves their action unchanged. It introduces new critical points at $r$ close to $1$. These have action $\pacal=\alpha(H_s'-H_s)+\beta\approx \alpha H_s'=T$ and are therefore at all times close to $\Scal$ and therefore far from the boundary of the action window. This is a preparatory step such that the mirroring in Step 5 is smooth. This perturbation is small.

\textbf{Step 5} directly homotopes the functional rather than the triple $(K,\frac 1\mu\alpha,\beta)$. The homotopy is made such that 
$$\pacal_s(x,\eta)=\int_0^1\lambda(\dot x) - (1-s)\frac 1\mu\alpha K- s|K|\;dt +\beta.$$
In total this lifts up the part left of the terrace point, converting the terrace to a minimum and introducing the desired V-shape for the Hamiltonian. This movement is compensated by changing $\eta$ to 1. This perturbation is large. 
During the deformation critical points can be characterized as Reeb chords from $L$ to $L$ of period $T$ together with numbers $r,\eta$ such that 
\begin{equation}
\label{Pcrit''}
	\begin{cases}
		T=(1-s)\frac1\mu\alpha(\eta)K'(r) + s|K|'(r),\\
		(1-s)\frac1\mu\dot\alpha(\eta)K(r)=\dot\beta(\eta).
	\end{cases}
\end{equation}
The action at a critical point equals 
$$\pacal_{H,\alpha,\beta}(x,\eta)=(1-s)\frac1\mu\alpha(\eta)(K'(r) - K(r))+s\frac1\mu\alpha(\eta)(|K|'(r) - |K|(r))+\beta(\eta).$$

In the following we use that $a,b$ are far from $\Scal$. It is sufficient to show that $r\sim 1$, because then $K(r)\sim |K|(r)\sim 0$ and thus from the first equation~(\ref{Pcrit''}) we have $\pacal_s(x,\eta)\sim T$. But then all critical actions are close to $\Scal$ and therefore far from $a,b$. We distinguish several cases.

[$\eta=0$]: Then the second equation implies that $K=0$, thus $r=1$. Furthermore we see that $K'=0$ and by the first equation $T=0$ and we have a constant orbit.

[$\eta>0$]: Then $\dot\beta(\eta)>0$ and thus by the second equation $r>1$, which implies $K=|K|$. The following are subcases.

[$\eta>0$, $s$ is close to $1$]: Then $(1-s)\frac1\mu\alpha(\eta)+s\sim1$ and the first equation tells us that $T\sim K'$, but since the asymptotic slope $\mu$ of $K$ is far from $\Scal$, this implies $r\sim1$.

[$\eta>0$, $s$ is far from $1$, $\alpha(\eta)=\eta$]: The second equation $(1-s)\frac1\mu K(r)=\dot\beta(\eta)$ tells us that $K(r)\sim 0$, hence $r\sim 1$.

[$\eta>0$, $s$ is far from $1$, $\alpha(\eta)\neq\eta$]: Then $\alpha(\eta)\sim A=\mu$. Then the first equation tells us that $T\sim K'$, but since the asymptotic slope $\mu$ of $K$ is far from $\Scal$, this implies $r\sim 1$. This finishes the case $\eta>0$.

[$\eta<0$]: Then $\dot\beta(\eta)<0$ and thus by the second equation $r<1$, which implies $-K=|K|$. The following are subcases.

[$\eta<0$, $s$ is close to $1$]: Then $(1-s)\frac1\mu\alpha-s\sim-1$ and by the first equation $T\sim -K'=|K|'$. Since $\mu$ is far from $\Scal$ this implies that either $r\sim 1$ or that $r<1-\frac12\delta$ (where $H$ is bent into the constant $-\frac23\delta$). In the first case we are done. The second case implies that $|K(r)|\sim \frac12\delta\mu$ and with $(1-s)\frac1\mu\alpha-s\sim-1$ we get $\pacal\sim -K'+K\leq K\leq-\frac12\delta\mu$. By our choice of $\mu$ this is well out of the action window.

[$\eta<0$, $s$ is far from $1$, $\alpha(\eta)=\eta$]: Then the second equation tells as that $K(r)\sim 0$, hence $r\sim 1$. 

[$\eta<0$, $s$ is far from $1$, $\alpha(\eta)\neq\eta$]: Then $\alpha(\eta)\sim -A=-\mu$ and thus $(1-s)\frac1\mu\alpha+s\sim -1$. Then the first equation tells us that $T\sim -K'=|K|'$, and thus $r$ is either close to $1$ or $\leq 1-\frac12\delta$. But then $\pacal_s\sim -K'+K$ and we conclude as in the case when $s$ is close to 1.

In conclusion in all cases the action of all critical points lies either close to $\Scal$ or outside $[a,b]$ and therefore far from $a,b$. All the above estimates are quantified rigorously in~\cite[Section 6, Step 4]{CFO10}.

\textbf{Step 6} can already be performed in the framework of Section~\ref{sec:wfh}. It replaces $|K|$ by $|K|-\varepsilon$ and perturbs $|K|$ on $\{r\leq\frac \delta2\}$ and also at the minimum at $1$ such that the functional is really Morse and not just Morse--Bott. The resulting Hamiltonian is in $\widecheck\Hcal$ and has the form for which $\FH^{(a,b)}$ coincides with $\WcheckH^{(a,b)}$.

\subsubsection*{Compactness of moduli spaces} 
To get a well defined homology, the moduli spaces of solutions of the Floer equations with specified asymptotics must be compact modulo breaking. This follows from $C^\infty_{\rm loc}$ compactness of the spaces $\tiilde \Mcal((x,\eta),(x',\eta'),H,J)$. For this it is enough to show $L^\infty$-bounds for $x$ and $\eta$ and $L^\infty$-bounds for the first derivatives of $x$. Using the Floer equation~(\ref{Pfloer}) one then also has an $L^\infty$-bound for~$\eta'$. Then bootstrapping~(\ref{Pfloer}) yields $L^\infty$-bounds on higher derivatives.

The bound for the first derivatives of $x$ follows as always by a bubbling analysis since $[\omega]$ vanishes on $\pi_2(W)$ (because $\omega$ is exact) and on $\pi_2(W,L)$ (by hypothesis). 

\subsubsection*{$L^\infty$-bound on $\eta$}
For the situation before Step 1 this is classic. After Step 1 the bound of before still holds by our choice of $A$. 

In Step 2 we have $\beta=0$ and for $|\eta|\geq A$ we have $\dot\alpha(\eta)=0$, so the equation for $\eta$ becomes $\partial_s\eta(s)=0$ for $|\eta|>A$ and thus $|\eta|$ cannot exceed $A$. 

In Steps 3 to 6 the equation for $\eta$ is $\partial_s\eta(s)=-\dot\beta(\eta(s))$ for $|\eta|>A$. Since ${\rm sign}\,\dot\beta(\eta)={\rm sign}\,\eta$, it is clear that $|\eta|$ is bounded by $A$.

\subsubsection*{$L^\infty$-bound on $x$}
%

The proof follows the usual pattern. Since the energy of the Floer strip $(x_s,\eta_s)$ is bounded, the gradient $\nabla\pacal(x_s,\eta_s)$ can be large only for a finite time. One shows separately that a small gradient implies a bound on $r\circ u(s,t)$ and that in the finite time where the gradient is large, $r\circ u(s,t)$ cannot grow too much. Note that there are positive constants $A,A',B,C$ such that at each moment of the whole process we have the bounds
\begin{equation}\label{conditions}
\begin{cases}
	H=h(r)=Ar + A' \mbox{ for all }r\geq 2,\\
	\|\dot\alpha\|_{L^\infty}\leq B <\infty,\\
	\|\dot\beta\|_{L^\infty}\leq C <\infty.
\end{cases}
\end{equation}
We also set $D=\min|H|$.

The following fundamental property says that at values of $s$ where $\nabla\pacal(x_s,\eta_s)$ is small, the radius $r$ stays bounded. This allows us to restrict our attention to the region where $\nabla\pacal$ is large, but this region must be compact since the energy is finite. 
\begin{equation}
	\forall\varepsilon>0\;\exists S\mbox{ such that } \|\nabla \pacal_{H,\alpha,\beta}(x,\eta)\|\leq \varepsilon \Rightarrow \max_{t\in[0,1]} r\circ x(t)\leq S.
\label{fundamental}
\end{equation}
This property holds during all the Steps $2$--$6$, as is shown in~\cite[Lemma 4.7]{CFO10}. The proof holds verbatim in our situation. 

Now we can analyze the radial coordinate $r:M\times\RR\to\RR;(x,r)\to r$ along a local solution $(x_s(t),\eta_s)=(u(s,t),\eta(s))$ of the Floer equation. The crucial observation is the following estimate for the Laplacian.
\begin{lemma} 
	If $H=h(r)$ depends only on $r$, a local solution $(u(s,t),\eta(s))$ of~(\ref{Pfloer}) satisfies at image points in $M\times\RR\subset W$ the bounds 
	\begin{align*}
		\Delta (r\circ u)&= \|\partial_s u\|^2-\partial_s(h'(r)\alpha(\eta)) (r\circ u),\\
		\Delta (\log r\circ u)&\geq -\partial_s(h'(r)\alpha(\eta)).
	\end{align*}
\end{lemma}
	This is Lemma~4.1 in~\cite{CFO10}, and the proof is not affected by the change to the open string situation.

If $(H,\alpha,\beta)$ satisfies furthermore~(\ref{conditions}), then we obtain for $r\circ u\geq2$ the bound 
\begin{equation}\label{laplaceestimate}
	 \Delta(\log r\circ u)\geq -A^2B^2D-ABC
\end{equation}
as in~\cite[Lemma~4.2]{CFO10}. 

To invoke the maximum principle, we will need that at the sets $\{t=0\}$ and $\{t=1\}$ (which get mapped to $L$), the function $r\circ u$ satisfies the Neumann condition. At $t=0$ we compute
\begin{align}
	\partial_t (r\circ u(s,t))|_{t=0}&=\langle \nabla r,\partial_t u(s,0)\rangle\label{neumann}\\
	&=\langle \nabla r, J[\partial_s u(s,0)-J\alpha(\eta)X_H]\rangle\nonumber\\
	&=\omega(\nabla r,\partial_s u(s,0)-J\alpha(\eta)X_H)=0.\nonumber
\end{align}
The last equality holds since $JX_H$ and $\nabla r$ are parallel, and since $\nabla r,\partial_su(s,0)\in T_{x(s,0)}L$ which is Lagrangian. For $t=1$ the computation is identical. It also holds for the function $\log r\circ u$. 
\begin{lemma}[$L^\infty$-bound on $x_s$]
	Suppose that $(H,\alpha,\beta)$ satisfies conditions~(\ref{conditions}) and~(\ref{fundamental}). Also suppose that $A\notin\Scal$. Consider a negative gradient flow line $(x_s(t),\eta_s)=(u(s,t),\eta(s))$ with is asymptotic to the critical points $(x_1,\eta_1)$ and $(x_2,\eta_2)$. Let $\Ecal:=\Ecal(u,\eta):=\pacal(x_1,\eta_1)-\pacal(x_2,\eta_2)$ be the energy of this flow line. Then for all $\varepsilon$ and the corresponding $S$ from~(\ref{fundamental}) and for all $(s,t)$ we have the following estimate
	$$\log r\circ u(s,t)\leq \max (\log 2,\log S) + \frac{(A^2B^2D+ABC)\Ecal(u,\eta)^2}{2\varepsilon^4}.$$	
\end{lemma}
\begin{proof}
		
	We follow~\cite[Proposition 4.3]{CFO10}. Fix $s_0$. Let $[s_-,s_+]$ be the maximal compact interval containing $s_0$ such that 
	$$\forall\sigma\in[s_-,s_+]: \|\nabla\pacal(x_\sigma,\eta(\sigma))\|\geq\varepsilon.$$
	It is possibly empty, but it exists since 
	$$\Ecal=\int_\RR\|\nabla\pacal(x_s,\eta_s)\|^2\;ds=\pacal(x_1,\eta_1)-\pacal(x_2,\eta_2)$$
	is finite and $s_+-s_-\leq\frac\Ecal{\varepsilon^2}$. Because of~(\ref{fundamental}) we have $r\circ x_{s_\pm}(t)<S$ $\forall t\in[0,1]$.
	
	On the strip $[s_-,s_+]\times[0,1]$ the function 
	$$\chi=\log r\circ u + \frac12(A^2B^2D+ABC)(s-s_0)^2$$
	is subharmonic at places where $r\circ u\geq 2$ because of~(\ref{laplaceestimate}). Define the set
	$$\Omega=\{(s,t)\in[s_-,s_+]\times[0,1]\mid \log r\circ u\geq\max(\log 2,\log S)\}.$$ 
	At the boundary $\partial \Omega$ either $\log r\circ u\leq \max(\log 2,\log S)$ or $\log r\circ u$ satisfies the Neumann condition~(\ref{neumann}). By the maximum principle we conclude that $\chi$ has no maximum in the interior of $\Omega$ nor at a boundary point where the Neumann condition holds. We conclude that for $(s,t)\in\Omega$,
\begin{align*}
	\log r\circ u(s,t) &\leq\chi(s,t)\leq \max_{(s,t)\in\Omega} \chi(s,t)\\
	&\leq \max(\log 2,\log S) + \frac12(A^2B^2D+ABC)(\max\{s_0-s_-,s_+-s_0\})^2\\
	&\leq \max(\log 2,\log S) + \frac12(A^2B^2D+ABC)(s_+-s_-)^2\\
	&\leq \max(\log 2,\log S) + \frac12(A^2B^2D+ABC)\left(\frac\Ecal{\varepsilon^2}\right)^2,
\end{align*}
which is the desired bound.
\end{proof}

\subsection{Time-dependent Lagrangian Rabinowitz--Floer homology (TH)}~\label{sec:tlrfh}
In contrast to the homologies we encountered so far, the aim of TH is not to capture the dynamics of the Reeb flow of $\lambda|_M$, but rather to set up a tool that allows us to study positive contactomorphisms. This will result in a homology TH which is an invariant of a Hamiltonian $H$ on $\haat W$. We will find a uniform way to construct such a Hamiltonian from a positive contact Hamiltonian $h^t$ such that TH becomes an invariant of $h^t$. 

\subsubsection*{The action functional}
Let $H^t:W\times [0,1]\to \RR$ be a time-dependent smooth function on $W$. We define the action functional $\tacal_{H^t}:\Pcal(L)\times \RR\to\RR$ by 
\begin{eqnarray}\label{functional}
	\tacal_{H^t}(x,\eta)&=&\frac1\kappa \left(\int_0^1x^*\lambda-\eta\int_0^1H^{\eta t}(x(t))\;dt\right),
\end{eqnarray}
where $\kappa$ is some positive constant discussed later. A pair $(x,\eta)$ is a critical point of $\tacal_{H^t}$ if and only if it satisfies the equations
\begin{equation*}\label{crit:lrfh}
\left\{\begin{array}{rcl}
		\dot x(t)&=&\eta X_{H^t}(x(t)),\\
		H^\eta(x(1))&=&0,
	\end{array}\right.
\end{equation*}
where $X_{H^t}$ is the Hamiltonian vector field generated by $H^t$.
The first equation implies that $x$ is an orbit of $X_{H^t}$, but with time scaled by $\eta$. The second equation, which one deduces by partial integration, implies that the orbit ends on $(H^{\eta})^{-1}(0)$. If $H^t=H$ is autonomous, then (up to the constant $\kappa$) the functional $\tacal_H$ is as for autonomous Rabinowitz--Floer homology $\AH$. Then $H^{-1}(0)$ is a hypersurface for which $\eta$ plays the role of a Lagrange multiplier, and $H(x(t))=0$ for all $t$. For time-dependent $H^t$, however, there is no such hypersurface, and $H^{\eta t}(x(t))$ might be very large or very small for $t<1$.

\subsubsection*{Construction of the Hamiltonians}
For the study of positive contactomorphisms~$\varphi$ it is crucial to carefully construct Hamiltonians~$H^t$ on $W$ in such a way that critical points of $\tacal_{H^t}$ encode dynamical information on $\varphi$, such that the resulting TH is well-defined and such that for monotone deformations of $H^t$ the continuation morphisms are monotone with respect to the action. 

We start with the object we actually want to study: a positive contactomorphism~$\varphi$ of $(M,\alpha)$. By definition there is a positive path of contactomorphisms $\varphi^t$ such that $\varphi^0=id,\;\varphi^1=\varphi$. By the second part of the proof of \cite[Proposition 6.2]{FLS15}, $\varphi^t$ can be deformed with fixed endpoints to $\tilde\varphi^t$ such that $\tilde\varphi^t$ is the Reeb flow of $\alpha$ for $t$ near $0$ and $1$. This means that the contact Hamiltonian $h^t$ on $M$ generating~$\varphi^t$ is constant $\equiv 1$ for $t$ near $0$ and $1$. This deformation can be performed such that the order is preserved: $\tilde h_1^t\leq\tilde h_2^t$ whenever $h_1^t\leq h_2^t$. From now on we assume that this deformation is already performed if not stated differently. 
Then the concatenation of positive contactomorphisms is a positive contactomorphism. More specifically, $h^t$ permits smooth periodic or constant extensions to $t\in\RR$. We choose the extension $h^t\equiv1$ for $t\leq0$, which later will guarantee monotonicity of continuation morphisms. For positive time we choose $h^t$ for every integer step $t\in[k,k+1]$ individually such that $h^t$ and its derivatives have uniform bounds. In the actual applications the choice will be that $h^t$ is periodic for $t\geq0$, see Figure~\ref{fig:ht}. Later on we will use the interval $[0,1]$ to encode the information that we count orbits from our base Legendrian to another Legendrian, and then extend $h^t$ to a periodic function for $t\geq 1$. We can now update Assumption~\ref{f=0} to the following analogon for the new situation:

\begin{figure}[h]
	\centering{
		\includegraphics{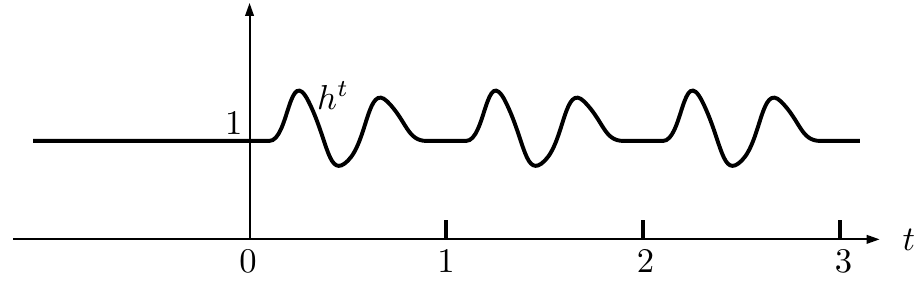}
	}
	\caption{The function $h^t(x)$ at a given point $x\in M$}
	\label{fig:ht}
\end{figure}

\begin{assumption}\label{classofh}
	The pair $(W,L)$ consists of a Liouville domain $(W,\omega,\lambda)$ with contact boundary $(M,\xi=\ker\lambda|_M)$ and an asymptotically conical exact Lagrangian~$L$ with connected Legendrian boundary $\Lambda=\partial L$ such that $\lambda|_L=0$, and such that $[\omega]|_{\pi_2(W,L)}=0$. The contact Hamiltonian $h^t:M\times\RR\to\RR$ satisfies for some positive constants $c,C,c'$
	\begin{itemize}
		\item $h^t\equiv 1$ for $t\leq0$,
		\item $0<c\leq h^t\leq C$ and $|\frac d{dt}h^t|\leq c'$,
		\item  $\bigcup_{t\neq 0}\varphi^t_{h^t}\Lambda$ and $\Lambda$ intersect transversely.
	\end{itemize}
\end{assumption}
The third assumption implies that the action functional $\tacal_{H^t}$ for $h^t$ defined by~(\ref{primitivesH}) and~(\ref{functional}) is Morse away from $\eta=0$.

To work in the Liouville domain we construct from the contact Hamiltonian $h^t$ on $M$ a Hamiltonian $H^t$ on $W$ in a uniform way, depending on two large enough parameters $\kappa$ and $K$, which we choose for every finite action window individually. The flow lines of $X_{H^t}$ are lifts of $\varphi^t$-flow lines if we set on $M\times\RR^{> 0}$
\begin{equation}
H^t:=rh^t-\kappa.
\label{primitivesH}
\end{equation}
 For such an $H^t$ the critical points of $\tacal_{H^t}$ end in $\{r=\kappa/h^\eta\}$. Changing $\kappa$ does not change critical points in an essential way (provided they do not run into $r=0$), but translates them in the $r$-direction. In order to have a smooth Hamiltonian on $\haat W$ and to get compactness of moduli spaces later on, we deform $H^t$ to depend only on $r$ but not on $t$ for $r\leq1$ and $r\geq\kappa K$, see Figure~\ref{fig:Hr}. To make this precise we choose independently of the action window a constant $\mu\geq\max\{h^t(x)\mid x\in M\}$, a convex smooth function $\rho:\RR^{\geq 0}\to\RR,$ that will play the role of the radius smoothed out over $W$, and a smooth function $\beta:\RR^{\geq0}\to[0,1]$, that will serve as a transition parameter, such that

\begin{figure}[h]
	\centering{
		\includegraphics{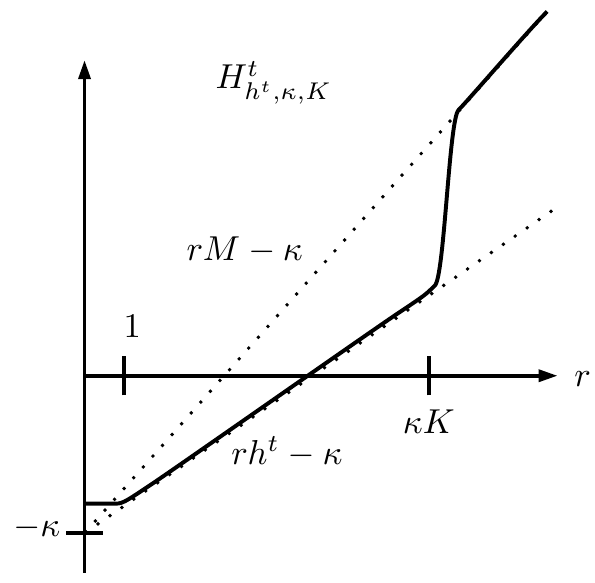}
	}
	\caption{The function $H^t_{h^t,\kappa,K}$ at a fixed time $t$ on a line $\RR^{\geq 0}\times \{x\}$ in dependence of $r$.}
	\label{fig:Hr}
\end{figure}

\begin{eqnarray*}
	\rho(r)&=&\begin{cases} 1-\frac23\delta &\quad\mbox{if}\quad r\leq1-\delta,\\r&\quad\mbox{if}\quad r\geq 1,\end{cases}\\
	\beta(r)&=&\begin{cases}0&\quad\mbox{if}\quad r\leq1-\delta \quad \mbox{or}\quad r\geq \kappa K+1,\\1&\quad\mbox{if}\quad 1\leq r\leq \kappa K.\end{cases}
\end{eqnarray*}
Then we define the Hamiltonian
\begin{eqnarray*}
	H^t_{h^t,\kappa,K}(x,r)&=&\rho(r)\Big(\beta(r)\,h^t(x)+(1-\beta(r))\,\mu\Big)-\kappa.
\end{eqnarray*}
The factor $\frac1\kappa$ in Definition~(\ref{functional}) does not influence the critical points, but only their action values. In fact, the following lemma shows that for $\kappa,K$ large enough, the critical points (up to translation in the $r$-direction) and their actions do not depend on the choice of the constants.

\begin{lemma}\label{independence}
	Let $h^t$ satisfy Assumption~\ref{classofh}. Given $a<b$, there are constants $\kappa_0,K_0$ such that for $\kappa\geq\kappa_0$ and $K\geq K_0$ the following holds. If $(x,\eta)$ is a critical point with $a\leq\tacal_{h^t,\kappa,K}(x,\eta)\leq b$, then the radial component of $x$ stays in $[1,K\kappa]$ for $t\in[0,1]$ and $\tacal_{h^t,\kappa,K}(x,\eta)=\eta$. 
\end{lemma}
\begin{proof} 
	A detailed proof and in particular explicit choices for the constants $\kappa_0,K_0$ are given in~\cite[Proposition 4.3]{AF12} in the setup of cotangent bundles. It applies verbatim in the present setting.
\end{proof} 

In the following we abbreviate $\tacal_{h^t,\kappa,K}=\tacal$.

\begin{remark}
	For $h^t\equiv 1$ one can choose for all action windows $K=\kappa=\mu=1$. The function $H^t$ coincides with $H$ in the definition of $\acal$ in~(\ref{acal}). Thus, the functional $\tacal$ coincides with $\acal$, and AH is a special case of TH.
\end{remark}

\subsubsection*{The differential}
The differential is constructed exactly as in the autonomous case: Choose an asymptotically conical almost complex structure $J$ to define an $L^2$-metric with respect to which one considers negative gradient flow lines that correspond to solutions of a perturbed Cauchy--Riemann equation. We get transversality of moduli spaces by perturbing $J$. The desired $L^\infty$-bounds on the flow lines follow as before since for $r\geq K\kappa+1$ the Hamiltonian $H^t_{h^t,\kappa,K}$ is autonomous and linear in $r$. Thus, our choice of deforming $H^t_{h^t,\kappa,K}$ to an autonomous Hamiltonian for large~$r$ guarantees compactness. The drawback is that the resulting homology counts the orbits of $h^t$ only in the chosen action window. 

\subsubsection*{Action windows and definition of the homology}
As Lemma~\ref{independence} shows, the definition of TH must be done for finite action windows first, and then is extended. To do this we  choose $\kappa_0, K_0$ so large that Lemma~\ref{independence} holds for critical points with action in~$[a,b]$. We first generate the chain complex $\TC^b=\TC^b(h^t,\kappa,K)$ by the critical points of $\tacal$ with action $\leq b\in\RR$ and then define $\TC^{(a,b)}=\TC^b/\TC^a$ and its homology $\TH^{(a,b)}=\TH^{(a,b)}(h^t,\kappa,K)$. These groups are independent of $\kappa\geq\kappa_0,K\geq K_0$, which is why we denote them by $\TH^{(a,b)}(h^t)$ for brevity. 

For $a\leq a', b\leq b'$ there are (for $\kappa,K$ large enough) homomorphisms induced by inclusion of generators $\TC^{(a,b)}\to\TC^{(a',b')}$. We define $\TC^{(-\infty,b)}$ as the inverse limit, $\TC^{(a,\infty)}$ as the direct limit and $\TC=\TC^{(-\infty,\infty)}$ as direct inverse limit (in this order, to preserve exactness of long exact sequences), while adjusting $\kappa,K$. It still holds that $\TC^{(a,b)}=\TC^{(-\infty,b)}/\TC^{(-\infty,a)}$. We denote by $\TC_+^T=\TC^{(0,T)}$ the positive part of the chain complex and by $\iota$ the homomorphisms $\TH_+^T\to\TH_+^\infty$ induced by inclusion.

\subsubsection*{Invariance properties}
Consider now a family of Hamiltonians $h_s^t$ such that $\partial_sh_s^t$ is supported in $s\in[0,1]$. Suppose that for the associated family of functionals $\tacal_s:=\tacal_{h^t_s,\kappa,K}$ the constants $\kappa,K$ are chosen uniformly large enough for $[a,b]$. We set $\tacal_-=\tacal_s$ for $s\leq0$ and $\tacal_+=\tacal_s$ for $s\geq1$. The continuation homomorphism $\Phi:\TC(\tacal_-)\to\TC(\tacal_+)$ is defined as in the definition of the differential by counting the 0-dimensional components of the moduli space of curves $(x_s,\eta_s)$ that satisfy the equation
\begin{eqnarray}\label{conteqn}
	\partial_s(x_s,\eta_s)&=&-\nabla\tacal_s(x_s,\eta_s),
\end{eqnarray}
such that $\lim_{s\to\pm\infty}(x_s,\eta_s)=(x_\pm,\eta_\pm)$ for critical points $(x_\pm,\eta_\pm)$ of $\tacal_\pm$. Then $\Phi$ induces an isomorphism $\TH(\tacal_-)\to \TH(\tacal_+)$, because $\eta$ is bounded along deformations, and actually does not depend on the homotopy $h_s$ but only on the endpoints~$h_\pm$.

Unfortunately, this isomorphism does not respect the action the filtration of the homology. We therefore restrict our attention to monotone deformations, i.e.\ $\partial_sh_s^t(x)\geq0\;\forall\; s,t,x$. The following proposition says that such monotone deformations are compatible with the action filtration. In the proof it becomes clear why we extend $h^t$ to be constant for $t\leq0$.

\begin{proposition}[Monotonicity]\label{monotonicity}
	Let $(W,L)$ be a Liouville domain with asymptotically conical exact Lagrangian with $\lambda|_L=0$ such that $[\omega]|_{\pi_2(W,L)}=0$. Let $h_-^t\leq h_+^t$ be two time-dependent positive contact Hamiltonians that satisfy Assumption~\ref{classofh}. Then the continuation homomorphism 
	$$\Phi:\TH(h_-^t)\to \TH(h_+^t)$$
	restricts for every $a$ to 
	$$\Phi|_{\TC^a(h_-^t)}:\TH^a(h_-^t)\to \TH^a(h_+^t).$$
\end{proposition}
\begin{proof}
	It suffices to show that the action is non-increasing along solutions of~(\ref{conteqn}).\\

For the deformation from $h_-^t$ to $h_+^t$, define a monotone smooth function $\chi:\RR\to[0,1]$ such that
\begin{eqnarray*}
	\chi(s)&=&\begin{cases}0&\quad\mbox{if}\quad s\leq0,\\
	1&\quad\mbox{if}\quad s\geq 1,\end{cases}
\end{eqnarray*}
and set $h_s^t:=h^t_- + \chi(s)(h_+^t-h_-^t)$. Denote by $H_s^t$ and $\tacal_s$ the associated Hamiltonians and functionals. The deformation satisfies 
$$\begin{array}{ccccccc}
\frac{d}{ds}H_s^t &=& \chi'(s)(H_+^t-H_-^t) &=& \chi'(s)\,\rho(r)\,\beta(r)(h^t_+-h_-^t) &\geq& 0.
\end{array}$$

 For every $(x,\eta)$ we have,
\begin{eqnarray*}
	\frac{\partial}{\partial s} \tacal_s(x,\eta)&=&\int_0^1-\frac\eta\kappa\chi'(s)(H_+^{\eta t}-H_-^{\eta t})(x(t)) \;dt.\nonumber
\end{eqnarray*}

 Now consider a solution $(x_s,\eta_s)$ of~(\ref{conteqn}). Set $E=\int_{-\infty}^\infty \|\partial_s (x_s,\eta_s)\|^2\;ds$ and $\tacal_\pm=\tacal_\pm(x_\pm,\eta_\pm)$. 
We calculate
\begin{eqnarray}
	\tacal_+&=&\tacal_-+\int_{-\infty}^\infty\frac d{ds}\tacal_s(x_s,\eta_s)\;ds\nonumber\\
	&=&\tacal_-+\int_{-\infty}^\infty\Big(\frac\partial{\partial s}\tacal_s\Big)(x_s,\eta_s)+\big\langle\nabla\tacal_s(x_s,\eta_s),\partial_s(x_s,\eta_s)\big\rangle \;ds\nonumber\\
	&=&\tacal_--E+\int_{-\infty}^\infty\int_0^1-\frac{\eta_s}\kappa\chi'(s)(H_+^{\eta t}-H_-^{\eta t})(x_s(t)) \;dt\;ds.\nonumber
\end{eqnarray}
 If $\eta_s\geq0$, then $-\frac{\eta_s}\kappa\chi'(s)(H_+^{\eta t}-H_-^{\eta t})(x_s(t))\leq0$. If $\eta_s\leq0$, then $h_+^{\eta t}=h_-^{\eta t}=1$ and thus $-\frac{\eta_s}\kappa\chi'(s)(H_+^{\eta t}-H_-^{\eta t})(x_s(t))=0$. It follows that $\tacal_+\leq\tacal_-$. 
\end{proof}

We use this proposition to show that along deformations, although the exponential growth of TH might change, its positivity is preserved. We accomplish this by comparing with a Reeb flow. 

	%

\begin{proof}[Proof of Proposition~\ref{continuation}]
Suppose for now that $h_0^t\leq h_1^t$. Proposition~\ref{monotonicity} shows that then the deformation morphism $\Phi$ restricts to $\Phi|_{\TC^T(h_0^t)}:
\TC^T(h_0^t)\to \TC^T(h_1^t)$. Furthermore $h_0^t=h_1^t$ for $t\leq 0$, thus $\tacal(h_0^t)$ and $\tacal(h_1^t)$ have the same critical points with non-positive action, and constant critical points $(x,\eta)$ with $\eta\leq0$ are solutions of~(\ref{conteqn}). Since the action is non-increasing along solutions of~(\ref{conteqn}) we get that 
$$\Phi|_{\TC^0(h_0^t)}\;:\;\TC^0(h_0^t)\;\to\; \TC^0(h_1^t)$$
 is a lower diagonal isomorphism. For the homomorphism $\Phi_*$ induced in the quotient we thus have
\begin{eqnarray*}
	\Phi_*(\TC_+^T(h_0^t))&=&\Phi(\TC^T(h_0^t))/\Phi(\TC^0(h_0^t))\\
	&=&\Phi(\TC^T(h_0^t))/\iota^{0,T}\TC^0(h_1^t)\\
	&\subseteq&\TC_+^T(h_1^t).
\end{eqnarray*}

Since $\Phi$ induces an isomorphism in $\TH^\infty$, abbreviating $\iota=\iota_{+}^{T,\infty}$, we conclude that 
\begin{eqnarray}\label{monotonicity1}                                                                             
	\dim \iota_*\TH_+^T(h_0^t)&\leq&\dim \iota_*\TH_+^T(h_1^t).
\end{eqnarray}

Now consider constants $c>0$. Unlike $h^t$, most constants are not equal to $1$ for $t$ near $0$ or $1$, so we need to modify them to fit our setup. From the proof of~\cite[Proposition~6.2]{FLS15} it is clear that for every $c$ there is a function $h_c^t:M\times[0,1]\to\RR$ with $h_c^t=1$ for $t$ near 0 and 1 and such that the induced flow $\varphi_{h_c^t}^t$ is a time-reparametrization of the Reeb flow $\varphi_1^t$ that satisfies $\varphi^1_{h_c^t}=\varphi^c_1$. The functions $h_c^t$ can be chosen continuously in $c$ such that $h_c^t\leq h_C^t$ if $c\leq C$. Extend $h_c^t$ constantly for $t<0$ and periodically for $t>0$. 
By construction of $h_c^t$ it is clear that there exists a monotone function $\tau_c:\RR^{\geq 0}\to\RR^{\geq 0}$ such that $\varphi_{h_c^t}^{\tau_c(t)}=\varphi_{1}^{t}$. For such a function we have $\tau_c(ck)=k$ for all $k\in\NN$ and thus $\lim_{t\to\infty}\frac{\tau_c(t)} t=\frac1c$.
The critical points of $\tacal_{1}$ with action $T$ are in bijection with the critical points of $\tacal_{h_c^t}$ with action $\tau_c(t)$. We deform $h_1^t=1$ to $h_C^t$ through $h_c^t$, while deforming the action window $(\varepsilon,T)$ with $\varepsilon,T\notin\Scal$ through $(\tau_c(\varepsilon),\tau_c(T))$. Since $\tau_c$ is injective, this deformation has no window crossing and we therefore have
\begin{eqnarray*}
	\iota_*\TH_+^{\tau_C(T)}(h_C^t)&\cong&\iota_*\TH_+^{T}(1).
\end{eqnarray*}
From Section~\ref{sec:sequence} we know that $\Gamma^{\rm symp}(W,L)$ is the exponential dimensional growth of~$\AH(W,L)$, thus also of $\TH(1)$. We conclude that for every $C$ the exponential growth of $\dim \iota_*\TH_+^T(h_C^t)$ is $C\Gamma^{\rm symp}(W,L)$.

Now choose $c,C>0$ such that $c\leq h^t\leq C$. The functions $h_c^t$ can be chosen such that still $h_c^t\leq h^t\leq h_C^t$ see Figure~\ref{fig:htsandwich}. 

\begin{figure}[h]
	\centering{
		\includegraphics{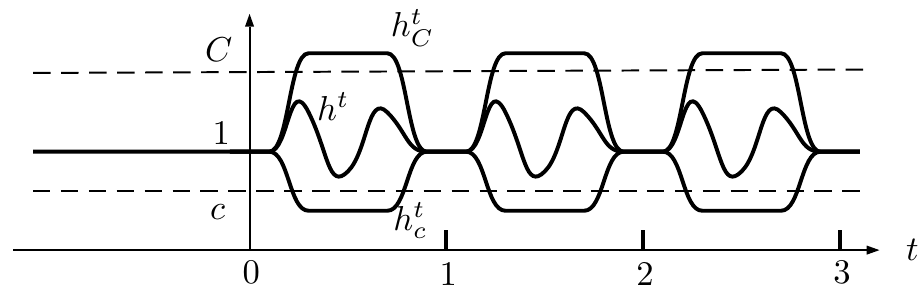}
	}
	\caption{The sandwiching of $h^t$ by $h_c^t$ and $h_C^t$.}
	\label{fig:htsandwich}
\end{figure}

We apply~(\ref{monotonicity1}) twice, first to a monotone deformation from $h_c^t$ to $h^t$ and then to a monotone deformation from $h^t$ to $h_C^t$. 
With~(\ref{monotonicity1}), this results in 
$$\dim \iota_*\TH_+^T(h_c^t)\leq\dim \iota_*\TH_+^T(h^t)\leq\dim \iota_*\TH_+^T(h_C^t).$$
Therefore, the quantitative bound claimed in Proposition~\ref{continuation} follows. In particular, for every positive path of contactomorphisms the growth of Rabinowitz--Floer homology is positive if and only if it is positive for a Reeb flow.
\end{proof}

Finally we show that from the homological growth of the number of chords from $\Lambda$ to $\Lambda$ we can deduce information about the growth the number of chords from~$\Lambda$ to a Legendrian~$\Lambda'$ that is isotopic to~$\Lambda$ through Legendrians. In~\cite{AM17} Alves and Meiwes showed the corresponding result for Reeb chords by constructing a $\WslantH(W,L)$-module structure on the wrapped Floer homology $\WslantH(W,L\to L')$, whose generators are Reeb chords from $L$ to nearby Lagrangians $L'$. In the following proof we do not need algebraic structures on Rabinowitz--Floer homology, but prove the result by just using continuation morphisms. Since the class of Reeb flows is included in the class of positive paths of contactomorphisms, one can replace the use of module structures in~\cite{AM17} by the following elementary proof.

	%
\begin{proof}[Proof of Proposition~\ref{changeLegendrian}]
	The idea of the proof is a rearrangement of information: If $\psi(\Lambda)=\Lambda'$, then the set of $\varphi^t$-chords from $\Lambda$ to $\Lambda'$ is in bijection with the set of $\psi^{-1}\circ\varphi^t\circ\psi$-chords from $\psi^{-1}\Lambda$ to $\Lambda$, see Figure~\ref{fig:LambdaToLambdaprime}.
	
	\begin{figure}[h]
	\centering{
		\includegraphics{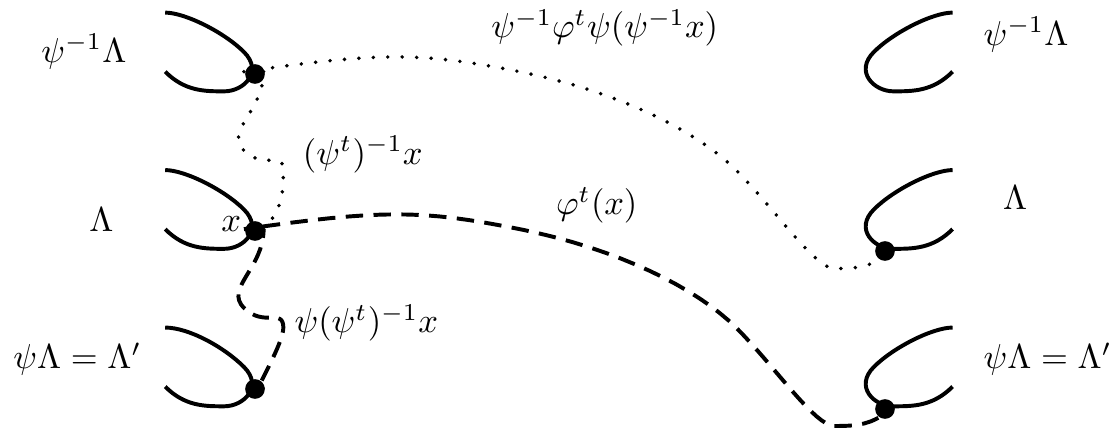}
	}
	\caption{A sketch of the geometric situation. A point $x$ is taken along the dotted line from $\Lambda$ to $\psi^{-1}\Lambda$ by $(\psi^t)^{-1}$ and then to $\Lambda$ by the flow of $\varphi^t$, conjugated by $\psi^{-1}$. Applying $\psi$ to everything yields the dashed line which, ignoring the first part from $\Lambda'$ to $\Lambda$, is a $\varphi^t$-chord from $\Lambda$ to $\Lambda'$.}
	\label{fig:LambdaToLambdaprime}
\end{figure}

We construct a positive path of contactomorphisms that is a Reeb flow for~$t\leq 0$, and the conjugate of $\varphi_{h^t}^t$ by $\psi^{-1}$ precomposed with $\psi^{-1}$ for $t\geq 1$ as follows. By the isotopy extension theorem we can extend the isotopy of Legendrians to a path of contactomorphisms $\psi^t,t\in[0,1],$ with $\psi^0=\id$ and $\psi^1=\psi$. Denote by $g^t$ the (not necessarily positive) contact Hamiltonian that generates $(\psi^t)^{-1}$. Denote by $\varphi^t,t\in\RR,$ the path of contactomorphisms generated by $h^t$ and denote by $\tilde h^t$ the Hamiltonian that generates $\tilde\varphi^t:=\psi^{-1}\circ\varphi^t\circ\psi$. 
Note that $\tilde\varphi^{kt},t\in\RR,$ is generated by $k \tilde h^{kt}$ and that $\tilde\varphi^{kt}\circ(\psi^t)^{-1},t\in[0,1],$ is generated by $\tilde g^t:=(k\tilde h^{kt})\#g^t=k\tilde h^{kt}+g^t\circ(\tilde\varphi^{kt})^{-1}$. 
We choose $k$ so large that $\tilde g^t\geq\max\{\tilde h^t\}$. 
Using a convex combination with a Reeb flow, we find a Hamiltonian $\overline g^t$ such that $\varphi_{\overline g^t}^1=\varphi_{\tilde g^t}^1$, such that $\overline g^t=1$ for $t$ close to 0 and $t$ close to 1, and such that $\overline g^t\geq \tilde h^t$.

Define $\overline h\vphantom{h}^t$ by
	\begin{equation*}
	\overline h\vphantom{h}^t=\begin{cases}
		1 &\mbox{ if }t<0,\\
		\overline g^t &\mbox{ if }t\in[0,1],\\
		\tilde h^t &\mbox{ if }t>1.
	\end{cases}
	\end{equation*}
This function is continuous at $t=0$ and $t=1$ because there $1=\overline g^t=\tilde h^t$. The induced positive path of contactomorphisms $\varphi^t_{\overline h\vphantom{h}^t}$ is a Reeb flow for $t\leq 0$, the time-$1$ map is~$\varphi_{\overline h\vphantom{h}^t}^1=\tilde\varphi^k\circ (\psi)^{-1}$ and for $t\geq 1$ we have $\varphi_{\overline h\vphantom{h}^t}^t=\tilde\varphi^{t-1+k}\circ (\psi)^{-1}$.

By our choice of $k$, $\overline h\vphantom{h}^t\geq\tilde h^t$ for all $t$. We can therefore deform $\tilde h^t$ to $\overline h\vphantom{h}^t$ in a monotone increasing way. This deformation thus induces a monotone morphism in $\TH^a$ by Proposition~\ref{monotonicity}. By Proposition~\ref{continuation}, the exponential dimensional growth of $\TH(\overline h\vphantom{h}^t)$ is still positive. The critical points $(x,T+1)$ of $\pacal_{\overline h\vphantom{h}^t}$ with $T+1>1$ are orbits that start on $\Lambda$, pass through $\tilde\varphi^k\circ\psi^{-1}(x(0))$ at time $1$ and then follow $\tilde\varphi^t$ for time $T$ to land on $\Lambda$. They are in bijection with $\tilde\varphi^t$-orbits from $\psi^{-1}(x(0))\in\psi^{-1}\Lambda$ to $\Lambda$ of length $k+T$, with a ``starting tail'' following $(\psi^t)^{-1}$ from $\Lambda$ to $\psi^{-1}\Lambda$. By application of $\psi$ and by forgetting the starting tail, they are in bijection with $\varphi^t$-orbits from $\Lambda$ to $\Lambda'$.

From this we conclude that the number of $\varphi^t$-orbits from $\Lambda$ to $\Lambda'$ with length between $k$ and $T$ is the dimension of the chain complex $\TC^{(1,T)}(\overline h\vphantom{h}^t)$, and thus grows at least as fast as the dimension of $\TH^{(1,T)}(\overline h\vphantom{h}^t)$, which by Proposition~\ref{continuation} has at least exponential growth $\min\{\overline h\vphantom{h}^t\}\Gamma^{\rm symp}(W,L)$. But $\min\{\overline h\vphantom{h}^t\}=\min\{\tilde h^t\}$ by construction. Define the function $f$ by $(\psi^{-1})^*\alpha=f\alpha$. We have $X_{\tilde h^t}(x)=D\psi^{-1}(X_{h^t}(\psi x))$ for every $x\in M$ and thus
\begin{equation*}
	\tilde h^t(x)=\alpha(X_{\tilde h^t}(x))=\alpha(D\psi^{-1}(X_{h^t}(\psi x)))=(\psi^{-1})^*\alpha(X_{h^t}(\psi x))=f(\psi x)h^t(\psi x).
\end{equation*}
Thus, $\min\{\tilde h^t\}\geq \min \{f\}\min \{h^t\}$, and the quantitative assertion in Proposition~\ref{changeLegendrian} follows.
\end{proof}

\section{Proof of Theorem~\ref{poscont}}\label{sec:proof}

The main result follows by composition of the results of the previous sections. The lower bounds on volume growth using a method introduced in~\cite{A17}.

\begin{proof}[Proof of Theorem~\ref{poscont}]
	Let $\varphi^t, t\in\RR$ be a positive path of contactomorphisms with $\varphi^0=id,\varphi^1=\varphi$ such that its underlying contact Hamiltonian $h^t$  is constant 1 for $t\leq0$ and 1-periodic for $t>0$. We will show that the exponential growth of $\Vol(\varphi^t\Lambda)$ is positive, where $\Vol$ is taken with respect to some well chosen Riemannian metric. By a theorem of Yomdin~\cite{Y} this volume growth provides a lower bound on the topological entropy. 
	
From Propositions~\ref{acuteischeck} and \ref{alrfhvwh} we deduce the following chain of isomorphisms
$$\WslantH_+^a(W,L)\cong \WcheckH_+^a(W,L)\cong \AH^{(0,a)}(W,L),$$
for all $a>0$ such that $a\notin\Scal$. Thus, the exponential growth of $\dim\AH^{(0,a)}(W,L)$ is $\Gamma^{\rm symp}(W,L)$. Proposition~\ref{continuation} shows that the exponential growth of $\dim \TH^{(0,a)}(h^t)$ is at least $c\Gamma^{\rm symp}(W,L)$, where $c=\min h^t>0$. 

Since $\Lambda$ is a Legendrian sphere, there is a tubular neighborhood $\Ncal= B^{n}\times \Lambda$ of $\Lambda$ in $M$ that is a product of a ball and the Legendrian spheres $\Lambda$. By isotopy extension each of the fibers $\Lambda'$ is the image of $\Lambda$ by a contactomorphism $\psi$ of $M$ isotopic to the identity such that $(\psi^{-1})^*\alpha=f\alpha$. After choosing a smaller neighborhood, one can assume that $\min f\geq 1-\varepsilon$ for a uniform $\varepsilon>0$. By Proposition~\ref{changeLegendrian} we see that for all fibers $\Lambda'$ the number of $\varphi^t$-chords from $\Lambda$ to $\Lambda'$ has growth at least $\gamma:=(1-\varepsilon)c\Gamma^{\rm symp}(W,L)$.

Now we choose our Riemannian metric $g$ such that $g$ orthogonally splits on $\Ncal= B^{n}\times \Lambda$. We show that $\Vol_n(\bigcup_{t\in[0,T]} \varphi^t\Lambda)$ has growth at least $\gamma$ since it cuts through $\Ncal$ many times. In the following we regard $\varphi^t$ as a map $\varphi(x,t):(\Lambda,\RR)\to \bigcup_{t\in\RR} \varphi^t\Lambda=\varphi(\Lambda\times\RR)$.

Let $\pi:\Ncal\to B^n$ be the projection to the fiber. Then by Sard's theorem there is a subset $B'\subset B^n$ of full measure such that the map 
$$P:\varphi^{-1}(\pi^{-1}(B'))\subset\Lambda\times\RR\to B';\; (x,t)\to \pi\circ\varphi(x,t)$$ 
has only regular values. At these points $P^{-1}(b)\cap\Lambda\times[0,T]$ is finite for every $T$, so we can consider its number of elements $n_b(T)$.
Note that $P$ being regular implies that the functional in the proof of Proposition~\ref{continuation} is Morse for all action windows $(1,T+1), T>0,$ so the corresponding $\TH$ is well defined for these action windows. Since $n_b(T)$ counts the number of $\varphi^t$-trajectories from $\Lambda$ to $b\times\Lambda\subset\Ncal$, $n_b(T)$ has growth at least $\gamma$. 
Since
$$\Vol_n(\varphi(\Lambda\times[0,T])\cap\Ncal\geq \int_{B'} n_b(T)\; d\Vol_{n},$$
and since the integrand on the right hand side uniformly has growth at least $\gamma$, we conclude that $\Vol_n(\bigcup_{t\in[0,T]}\varphi^t\Lambda)\geq \Vol_n(\bigcup_{t\in[0,T]}\varphi^t\Lambda)\cap\Ncal$ has growth at least $\gamma$. 

Now to relate the growth of $\Vol_n(\bigcup_{t\in[0,T]} \varphi^t\Lambda)$ to the growth of $\Vol_n(\varphi^t\Lambda)$, we note that $|X_{h^t}|_g\leq C$ is bounded in length from above by compactness of $M$. Thus we can estimate
$$\Vol_n(\varphi(\Lambda\times[0,T])\leq C\int_0^T \Vol_{n-1}(\varphi^t\Lambda)\;dt.$$
Since the left hand side has growth at least $\gamma$, the same holds for the right hand side. This is only possible if the integrand also has growth at least $\gamma$. We conclude that the volume growth of $\varphi^t\Lambda$, thus the volume growth of $\varphi$ on $M$, and thus the topological entropy of $\varphi$ are at least $c\Gamma^{\rm symp}(W,L)>0$, since $\varepsilon>0$ was arbitrary. 
\end{proof}

\begin{remark}
	Note that for the definition of symplectic growth we required that $(\lambda,L)$ is regular. In contrast we do not require regularity for $\varphi^t$. This is because we only consider Hamiltonians as in the proof Proposition~\ref{changeLegendrian}, which are perturbed in the interval $[0,1]$.
\end{remark}


\end{document}